\documentclass{amsart}
\usepackage[utf8]{inputenc}
\usepackage{amsmath}
\usepackage{amsfonts}
\usepackage{mathtools}
\usepackage{amssymb}
\usepackage[english]{babel}
\usepackage{amsmath,amsthm}
\usepackage{amsfonts}
\usepackage{fancyhdr}
\usepackage{setspace}
\usepackage{color}
\usepackage[percent]{overpic}
\usepackage{fix-cm}
\usepackage{graphicx}
\usepackage[all]{xy}
\usepackage{mathrsfs}
\usepackage{mathdots}
\usepackage{amsbsy}
\usepackage{leftidx}
\usepackage[hidelinks]{hyperref}
\usepackage{xcolor}
\usepackage{todonotes}
\usepackage{tikz-cd}
\usepackage[noadjust]{cite}
\usepackage[mark=o]{dynkin-diagrams}
\usepackage{array,multirow}
\usepackage{tikz}

\hypersetup{
    linktocpage,
    colorlinks,
    linkcolor={magenta!100!black},
    citecolor={cyan!100!black},
    urlcolor={yellow!100!black}
}

\theoremstyle{plain}
\newtheorem{thm}{Theorem}[section]
\newtheorem{lem}[thm]{Lemma}
\newtheorem{prop}[thm]{Proposition}
\newtheorem{cor}[thm]{Corollary}

\theoremstyle{definition}
\newtheorem{dfn}[thm]{Definition}
\newtheorem{rem}[thm]{Remark}

\theoremstyle{remark}

\newtheorem{ex}[thm]{Example}

\numberwithin{equation}{section}

\newcommand{\liep}{\mathfrak{p}}
\newcommand{\liek}{\mathfrak{k}}
\newcommand{\lieq}{\mathfrak{q}}
\newcommand{\lieh}{\mathfrak{h}}
\newcommand{\lieg}{\mathfrak{g}}

\newcommand{\liea}{\mathfrak{a}}

\newcommand{\g}{\mathsf{G}}
\newcommand{\h}{\mathsf{H}}
\newcommand{\ko}{\mathsf{K}}
\newcommand{\w}{\mathsf{W}}
\newcommand{\p}{\mathsf{P}}
\newcommand{\n}{\mathsf{N}}
\newcommand{\m}{\mathsf{M}}
\newcommand{\z}{\mathsf{Z}}
\newcommand{\bor}{\mathsf{B}}
\DeclareMathOperator{\SL}{\mathsf{SL}}
\DeclareMathOperator{\GL}{\mathsf{GL}}
\DeclareMathOperator{\PSL}{\mathsf{PSL}}

\newcommand{\rr}{\mathbb{R}}
\newcommand{\zz}{\mathbb{Z}}
\newcommand{\cc}{\mathbb{C}}

\newcommand{\bc}{\begin{center}}
\newcommand{\ec}{\end{center}}

\newcommand{\x}{\mathscr{X}}
\newcommand{\f}{\mathscr{F}}
\newcommand{\grass}{\mathscr{G}}

\newcommand{\cone}{\mathcal{L}_\rho}
\newcommand{\conexi}{\mathcal{L}_\Xi}
\newcommand{\calc}{\mathcal{C}}

\newcommand{\bg}{\partial\Gamma}
\newcommand{\bgc}{\partial^{2}\Gamma}

\newcommand{\gh}{\Gamma_{\tn{H}}}

\newcommand{\pos}{\mathbf{pos}}
\newcommand{\bfi}{\mathbf{I}}
\newcommand{\bfp}{\mathbf{p}}
\newcommand{\bfo}{\mathbf{\Omega}}

\newcommand{\Imin}{\mathbf{I}_{\mathrm{min}}}
\newcommand{\Inonmax}{\mathbf{I}_{\mathrm{nonmax}}}

\newcommand{\tn}{\textnormal}

\begin{document}

\title[Domains of discontinuity]{Anosov representations acting on homogeneous spaces: domains of discontinuity}
\author{León Carvajales and Florian Stecker}
\address{\newline
  León Carvajales \newline
  Universidad de la República \newline
  Facultad de Ciencias Económicas y de Administración \newline
  Instituto de Estadística \newline
  CNRS IRL IFUMI \newline
  e-mail: leon.carvajales@fcea.edu.uy \newline
  \newline
  Florian Stecker \newline
  Florida State University \newline
  Department of Mathematics \newline
  e-mail: math@florianstecker.net}

\thanks{L.C. acknowledges funding by the Agencia Nacional de Investigación e Innovación FCE\_3\_2020\_1 \_162840 and by the German Research Foundation (DFG) through the project 338644254 (SPP2026 Geometry at Infinity).
  This material is based upon work supported by the National Science Foundation under Grant No. DMS-1928930 while L.C. participated in a program hosted by the Mathematical Sciences Research Institute in Berkeley, California, during the Fall 2020 semester.
  F.S. received support from the Klaus Tschira Foundation, the RTG 2229 grant of the DFG and the European Research Council under ERC consolidator grant 614733.
  This work was supported by the DFG under Germany's Excellence Strategy EXC-2181/1 - 390900948 (the Heidelberg STRUCTURES Cluster of Excellence).
}

\begin{abstract}
  We construct open domains of discontinuity for Anosov representations acting on some homogeneous spaces, including (pseudo-Riemannian) symmetric spaces.
  This generalizes work of Kapovich-Leeb-Porti on flag spaces.
  Our results complement those of Gu\'eritaud-Guichard-Kassel-Wienhard, who constructed proper actions of Anosov representations.
  For Zariski dense Ano\-sov representations with respect to a minimal parabolic subgroup acting on some symmetric spaces, we show that our construction describes the largest possible open domains of discontinuity.
\end{abstract}

\maketitle
\setcounter{tocdepth}{1}
\tableofcontents

\section{Introduction}

Let $\g$ be a Lie group and $\x$ be a $\g$-homogeneous space.
A $(\g,\x)$-\textit{structure} on a manifold $M$ is an atlas of charts from $M$ to $\x$, whose transition maps extend to elements of $\g$.
The study of geometric structures on manifolds is an active field, bringing together researchers in topology, geometry, dynamics and other areas.
See e.g. Kassel \cite{KasICM} or Wienhard \cite{WieICM} and references therein for a picture of the state of the art of the field.

A way of constructing a geometric manifold (or orbifold) is to consider a discrete subgroup $\Xi<\g$ and a domain of discontinuity $\bfo_\Xi\subset\x$ for $\Xi$ (an open set on which $\Xi$ acts properly).
Then the quotient space $M:=\Xi\backslash \bfo_\Xi$ is an orbifold and naturally carries a $(\g,\x)$-structure.
The main goal of this paper is to construct open domains of discontinuity in a large class of homogeneous spaces, including pseudo-Riemannian symmetric spaces, for an important class of discrete subgroups of $\g$ called \textit{Anosov subgroups}.

This problem has a long history. 
For instance, a particularly interesting case is the study of convex co-compact subgroups $\Xi<\PSL_2(\cc)$: these act properly on the real hyperbolic $3$-space $\mathbb{H}^3$ and on an open domain of discontinuity $\bfo_\Xi$ of its visual boundary $\partial\mathbb{H}^3$.
This domain of discontinuity is just the complement of the limit set $\Lambda_\Xi\subset \partial\mathbb{H}^3$ of $\Xi$.
The interplay between the geometry of the hyperbolic $3$-manifold $\Xi\backslash \mathbb{H}^3$ and its conformal boundary $\Xi\backslash \bfo_\Xi$ proves to be very fruitful.
Notably, it plays a central role in the proof of Thurston's Hyperbolization Theorem.

More recently, the deformation theory of convex co-compact subgroups of rank one Lie groups has been generalized to higher rank Lie groups by Labourie \cite{Lab} and Guichard-Wienhard \cite{GW}, through the notion of \textit{Anosov representations}.
This is a stable class of discrete and faithful representations $\rho:\Gamma\to\g$, from a word hyperbolic group $\Gamma$ into a semi-simple Lie group $\g$ of non-compact type.
They come in different flavors, according to the choice of a non-empty subset $\theta\subset\Delta$ of simple roots of $\g$.
More concretely, a $\theta$-Anosov representation has a compact invariant \textit{limit set} $\Lambda_\rho^\theta$ in the flag manifold $\f_\theta$ associated to $\theta$, which continuously and equivariantly identifies with the Gromov boundary $\bg$ of $\Gamma$.
Guichard-Wienhard \cite{GW} and Kapovich-Leeb-Porti \cite{KLPdomains} used the hyperbolic nature of the $\Gamma$-action on $\Lambda_\rho^\theta$ to construct co-compact open domains of discontinuity $\bfo_\rho\subset \f_{\theta'}=\x$ for such representations, where $\theta'\subset\Delta$ is possibly different from $\theta$.
The topology of the quotient space $\Gamma\backslash \bfo_\rho$ has been studied by many authors \cite{AMTW,ALSHitOrb,CTTMax,DScompact,ADLquasiHit,DavNearly}, see Alessandrini-Maloni-Tholozan-Wienhard \cite{AMTW} for a detailed account.

In this paper, rather than focusing on a flag manifold $\f_{\theta'}$ we consider a general $\g$-homogeneous space $\x$, with only the assumption that the diagonal action $\g\curvearrowright (\f_\theta\times\x)$ has finitely many orbits.
This holds when $\x=\f_{\theta'}$, when $\x$ is a (not necessarily Riemannian) \textit{symmetric space} of $\g$ (see Section \ref{sec: cartan for symmetric} for definitions) and when $\x$ is \textit{spherical} (see Luna \cite{Luna}).

A simple consequence of our main result is the following corollary, which states that the set of points in $\x$ which are in ``generic position'' with respect to $\Lambda_\rho^\theta$ forms a domain of discontinuity for $\rho$.
More precisely, for $x\in\x$ we let $\h^x$ be the stabilizer in $\g$ of $x$ and $\mathscr{M}_\theta^x\subset\f_\theta$ be the union of open orbits of the action $\h^x\curvearrowright\f_\theta$.
That is, $\mathscr{M}_\theta^x$ is the set of flags which are in ``general position" with respect to $x$.
See Figure \ref{fig: intro} for a picture of how this set may look like in concrete examples.

\begin{cor}[See Corollary \ref{cor: if limit set contained in open orbits, then dod}]\label{cor: dod for open positions in intro}
Let $\theta$ be a non-empty subset of simple roots of $\g$ and $\x$ be a $\g$-homogeneous space so that the action $\g\curvearrowright (\f_\theta\times\x)$ has more than one but finitely many orbits.
Then for every $\theta$-Anosov representation $\rho:\Gamma\to\g$, the open set
\begin{equation}\label{eq: dod for open position in intro}
\{x\in\x: \Lambda_\rho^\theta\subset \mathscr{M}_\theta^x\}
\end{equation}
\noindent is a domain of discontinuity for $\rho$.
\end{cor}


When $\g=\PSL_2(\cc)$ and $\x=\partial\mathbb{H}^3$, the set (\ref{eq: dod for open position in intro}) precisely coincides with the domain of discontinuity for convex co-compact subgroups alluded above. 
More generally, if $\x=\f_{\theta'} $ then Corollary \ref{cor: dod for open positions in intro} was already known by \cite{GW,KLPdomains}.
Further, for $\g=\mathsf{PSO}(p,q)$, $\x=\mathbb{H}^{p,q-1}$ (the \textit{pseudo-Riemannian real hyperbolic space of signature $(p,q-1)$}), and $\theta$ corresponding to the stabilizer of an isotropic line, Corollary \ref{cor: dod for open positions in intro} was proved by Danciger-Gu\'eritaud-Kassel \cite{DGK1}.
In fact, in that setting the set (\ref{eq: dod for open position in intro}) coincides with the domain of discontinuity $\Omega^{p,q-1}$ associated to $\mathbb{H}^{p,q-1}$-\textit{convex co-compact} subgroups (see \cite{DGK1} and Example \ref{ex: dod Hpq} for details).
The observation that $\Omega^{p,q-1}$ can be described using a generalization of Kapovich-Leeb-Porti's formalism \cite{KLPdomains} initiated this project.

Another simple consequence of our main result is the following finer statement in the case that there exists a Cartan involution $\tau$ of $\g$ leaving $\h^x$ invariant (this is notably the case if $\x$ is a symmetric space).
For $x\in\x$ we let $\mathscr{N}_\theta^x $ be the union of all non-closed orbits of the action $\h^x\curvearrowright\f_\theta$.
Note that $\mathscr{M}_\theta^x\subset\mathscr{N}_\theta^x$.

\begin{cor}[See Corollary \ref{cor: dod if unique minimal position}]\label{cor: dod for non closed positions in intro}
Let $\x$ and $\theta$ be as in Corollary \ref{cor: dod for open positions in intro}, and assume furthermore that for some $x\in\x$ there exists a Cartan involution $\tau$ of $\g$ so that $\tau(\h^x)=\h^x$.
Then for every $\theta$-Anosov representation $\rho:\Gamma\to\g$, the open set
\begin{equation}\label{eq: dod for non closed position in intro}
\{x\in\x: \Lambda_\rho^\theta\subset \mathscr{N}_\theta^x\}
\end{equation}
\noindent is a domain of discontinuity for $\rho$.
\end{cor}

As an example, Corollary \ref{cor: dod for non closed positions in intro} applies to the picture on the right in Figure \ref{fig: intro}, but not to the one on the left.

The sets (\ref{eq: dod for open position in intro}) and (\ref{eq: dod for non closed position in intro}) can be empty in some situations and usually won't be maximal domains of discontinuity.
Because of this, we will prove a finer and more systematic result (see Theorems \ref{thm: dod for anosov in intro} and \ref{thm: dod for anosov in intro II} below), whose statement and proof is inspired by \cite{KLPdomains}.
Moreover, we will prove that in some situations this construction describes maximal open domains of discontinuity in $\x$ (Theorem \ref{thm: maximality in intro}), generalizing the flag manifold case \cite{SteIDEALS}.

Before going into the statement of our main results, let us observe that the action of a given infinite discrete subgroup $\Xi<\g$ on $\x$ can be proper itself.
In fact one has the following criterion which intuitively states that this happens if and only if, up to compact sets, the group $\Xi$ ``drifts away" from point stabilizers.

\begin{thm}[Benoist \cite{BenProperness}, Kobayashi \cite{Kob}]\label{thm: benoist kobayashi}
Let $\x$ be any $\g$-homogeneous space and $\Xi<\g$ be an infinite discrete subgroup.
Let $\mu$ be a Cartan projection of $\g$ and $\h=\h^o$ be the stabilizer in $\g$ of a basepoint $o\in\x$.
Then the action $\Xi\curvearrowright\x$ is properly discontinuous if and only if, for every $t\geq 0$, \begin{equation}\label{eq: benoist kobayashi}
\#\{\gamma\in\Xi: d(\mu(\gamma),\mu(\h))\leq t\}<\infty.
\end{equation}
\end{thm}

By applying Theorem \ref{thm: benoist kobayashi}, Gu\'eritaud-Guichard-Kassel-Wienhard \cite[Corollary 1.9]{GGKW} constructed examples of Anosov representations acting properly on some homogeneous spaces $\x$.
Roughly speaking, this works as follows: $\rho$ being $\theta$-Anosov means that $\alpha(\mu(\rho(\gamma)))$ grows coarsely linearly in the word length of $\gamma$ for every $\alpha\in\theta$ \cite{BPS,GGKW,KLPDynGeomCharacterizations}.
In particular, if $\mu(\h)$ is a subset of $\cup_{\alpha\in\theta}\ker(\alpha)$, Theorem \ref{thm: benoist kobayashi} implies that the action of $\Gamma$ on $\x$ through $\rho$ is proper.
The quotient $(\g,\x)$-orbifold $\Gamma\backslash \x$ is then called a \textit{Clifford-Klein form} of $\x$.
Our main result, the construction of domains of discontinuity, is of course only interesting if the $\Gamma$-action on $\x$ is not already proper.
Moreover, in our maximality theorem we assume that the action on $\x$ is not proper through a refinement of condition (\ref{eq: benoist kobayashi}).

\subsection{Main results}


Let $\h=\h^o$ be the stabilizer in $\g$ of a given basepoint $o\in\x$.
Fix also a non-empty subset $\theta\subset\Delta$ of simple roots of $\g$, and let $\p_\theta$ be the corresponding parabolic subgroup of $\g$ (which without loss of generality we assume to be self-opposite, see Subsection \ref{subsec: flags}).
We then have identifications $\f_\theta\cong \g/\p_\theta \tn{ and } \x\cong\g/\h.$

We will always assume that the diagonal action $\g\curvearrowright(\f_\theta\times \x)$ has finitely many orbits.
The $\g$-orbit of $(\xi,x)\in\f_\theta\times\x$ is called the \textit{relative position} between the flag $\xi$ and the point $x$, and is denoted by $\pos(\xi,x)$.
The set of relative positions $\g \backslash (\f_\theta\times \x)$ carries a natural partial order, called the \textit{Bruhat order}, which is given by inclusions of orbit closures.
In other words, $\bfp' \leq \bfp$ if there exist sequences $\xi_n \to \xi$ and $x_n \to x$ with $\pos(\xi_n, x_n) = \bfp$ for all $n$ and $\pos(\xi, x) = \bfp'$.
Informally speaking, ``moving down" in this partial order means to decrease the ``degree of genericity" of the relative position.
See Figure \ref{fig: intro} for pictures in some concrete examples (more examples will be discussed in Sections \ref{sec: relative positions} and \ref{sec: fat ideals}).
Our finiteness assumption on the set of relative positions implies that this indeed defines a partial order (see Lemma \ref{lem: partial order}).
When $\x=\f_{\theta'}$ is a flag manifold the Bruhat order has been well studied, see e.g. \cite[Proposition 2.2.13]{SteThesis} for a combinatorial description.
In our general setting, it seems no such general description is available, see Richardson-Springer \cite{RichSpri} for a particular case.

\begin{figure}

  \begin{center}
    
    \begin{tikzpicture}[scale=0.5]
      \begin{scope}[scale=0.8]
        \fill[blue!20,rounded corners] (2.5,-1.8) rectangle (7.5,1.8);
        \fill[black!20,rounded corners] (-0.0,-3.2) rectangle (10.0,-17);
        \fill[black!35,rounded corners] (0.3,-8.3) rectangle (9.7,-16.7);
        
        \draw[thick,blue] (3.5,-1.5) -- (6.5,1.5);
        \draw[thick,red] (3.5,1.5) -- (6.5,-1.5);
        \draw[thick,blue] (1.0,-6.5) -- (4.0,-3.5);
        \draw[thick,red] (1.0,-3.5) -- (4.0,-6.5);
        \draw[thick,blue] (6.0,-6.5) -- (9.0,-3.5);
        \draw[thick,red] (6.0,-3.5) -- (9.0,-6.5);
        \draw[thick,blue] (1.0,-11.5) -- (4.0,-8.5);
        \draw[thick,red] (1.0,-8.5) -- (4.0,-11.5);
        \draw[thick,blue] (5.95,-11.5) -- (8.95,-8.5);
        \draw[thick,red] (6.0,-11.55) -- (9.0,-8.55);
        \draw[thick,blue] (3.45,-16.5) -- (6.45,-13.5);
        \draw[thick,red] (3.5,-16.55) -- (6.5,-13.55);

        \draw[thick,->] (4.0,-2) -- (3.5,-3);
        \draw[thick,->] (6.0,-2) -- (6.5,-3);
        \draw[thick,->] (2.5,-7) -- (2.5,-8);
        \draw[thick,->] (7.5,-7) -- (7.5,-8);
        \draw[thick,->] (4.25,-7) -- (5.75,-8);
        \draw[thick,->] (5.75,-7) -- (4.25,-8);
        \draw[thick,->] (3.5,-12) -- (4.0,-13);
        \draw[thick,->] (6.5,-12) -- (6.0,-13);

        \fill[red] (4,1) circle (0.2);
        \fill[red] (2.5,-5) circle (0.2);
        \fill[red] (6.5,-4) circle (0.2);
        \fill[red] (2.5,-10) circle (0.2);
        \fill[red] (8.5,-9) circle (0.2);
        \fill[red] (5,-15) circle (0.2);

        \fill[blue] (4,-1) circle (0.2);
        \fill[blue] (1.5,-6) circle (0.2);
        \fill[blue] (7.5,-5) circle (0.2);
        \fill[blue] (7.5,-10) circle (0.2);
        \fill[blue] (2.5,-9.8) arc (90:270:0.2) -- cycle;
        \fill[red] (2.5,-10.2) arc (-90:90:0.2) -- cycle;
        \fill[blue] (5,-14.8) arc (90:270:0.2) -- cycle;
        \fill[red] (5,-15.2) arc (-90:90:0.2) -- cycle;
      \end{scope}
      
      \begin{scope}[shift={(12,-1)}]
        \fill[blue!20,rounded corners] (-2.5,-1.8) rectangle (12.5,2.2);
        \fill[black!20,rounded corners] (-0.5,-3.1) rectangle (10.5,-11.7);
        \fill[black!35,rounded corners] (2.2,-8.1) rectangle (7.8,-11.5);
        
        \draw[thick,blue] (0,0) ellipse (2 and 1.2);
        \draw[thick,blue] (5,0) ellipse (2 and 1.2);
        \draw[thick,blue] (10,0) ellipse (2 and 1.2);
        \draw[thick,blue] (2.5,-5) ellipse (2 and 1.2);
        \draw[thick,blue] (7.5,-5) ellipse (2 and 1.2);
        \draw[thick,blue] (5,-10) ellipse (2 and 1.2);
        
        \draw[thick,red] (-2.2,1.4) -- (2,1.7);
        \draw[thick,red] (3,-1.5) -- (7,1.5);
        \draw[thick,red] (8,-1.5) -- (12,1.5);
        \draw[thick,red] (0,-4.2) -- (4.5,-3.3);
        \draw[thick,red] (5.5,-6.5) -- (9.5,-3.5);
        \draw[thick,red] (2.5,-9.2) -- (7,-8.3);
        
        \fill[red] (1,1.6) circle (0.2);
        \fill[red] (6.6,1.23) circle (0.2);
        \fill[red] (10,0) circle (0.2);
        \fill[red] (3.4,-3.5) circle (0.2);
        \fill[red] (8.75,-4.08) circle (0.2);
        \fill[red] (4.35,-8.85) circle (0.2);
        
        \draw[thick,->] (1.0,-2) -- (1.5,-3);
        \draw[thick,->] (4.0,-2) -- (3.5,-3);
        \draw[thick,->] (6.0,-2) -- (6.5,-3);
        \draw[thick,->] (9.0,-2) -- (8.5,-3);
        \draw[thick,->] (3.5,-7) -- (4.0,-8);
        \draw[thick,->] (6.5,-7) -- (6.0,-8);        
      \end{scope}
    \end{tikzpicture}
  \end{center}

  \begin{center}

\caption{A picture of the Bruhat order in two examples for $\g=\PSL_3(\rr)$.
    In both cases, $\f_\theta=\f_\Delta$ is the space of full flags in $\rr^3$, whose elements are represented in red as a point inside a line in the projective plane.
    Elements of $\x$ are drawn in blue.
    On the left $\x=\f_\Delta$ and on the right $\x = \mathsf{PSL}_3(\rr)/\mathsf{PSO}(2,1)$, with its elements represented by the corresponding isotropic cone.
    Black arrows generate the Bruhat order in each case.
    The purple box represents $\mathscr{M}_\theta^x$, which consists of a single relative position for the picture on the left, and of three relative positions for the picture on the right.
    The light grey box is the ideal $\Inonmax$ in each case, and the dark grey box is the minimal fat ideal.}
    \label{fig: intro}
\end{center}
\end{figure}

A subset $\bfi$ of relative positions is called an \textit{ideal} if for every $\bfp\in\bfi$ and every $\bfp'\leq\bfp$ one has $\bfp'\in\bfi$.
Given such a set and a $\theta$-Anosov representation $\rho:\Gamma\to\g$ with limit set $\Lambda_\rho^\theta\subset\f_\theta$ we may define $$\bfo_\rho^\bfi:=\{x\in\x: \pos(\xi,x)\notin\bfi \tn{ for every } \xi\in\Lambda_\rho^\theta\}.$$
\noindent As an example, the set (\ref{eq: dod for open position in intro}) in Corollary \ref{cor: dod for open positions in intro} corresponds to the ideal $\Inonmax$ consisting of non-maximal relative positions, and the set (\ref{eq: dod for non closed position in intro}) in Corollary \ref{cor: dod for non closed positions in intro} corresponds to the ideal $\Imin$ of minimal positions.

The set $\bfo_\rho^\bfi$ is $\Gamma$-invariant and the fact that $\bfi$ is an ideal implies that it is open, but to guarantee properness of the action $\Gamma\curvearrowright\bfo_\rho^\bfi$ an extra condition on $\bfi$ is needed.
In the case $\x=\f_{\theta'}$, Kapovich-Leeb-Porti \cite{KLPdomains} introduced the notion of \textit{fat ideal}, and showed that for such ideals the $\Gamma$-action on $\bfo_\rho^\bfi$ is indeed proper.
In Section \ref{sec: fat ideals} we generalize this notion to our setting and then prove the following in Section \ref{sec: dod}.

\begin{thm}[See Theorem \ref{thm: dod for anosov}(1)]\label{thm: dod for anosov in intro}
Let $\theta\subset\Delta$ be a non-empty subset of simple roots of $\g$ and $\x$ be a $\g$-homogeneous space so that the action $\g\curvearrowright (\f_\theta\times\x)$ has finitely many orbits.
Let $\bfi\subset \g\backslash(\f_\theta\times\x)$ be a fat ideal.
Then for every $\theta$-Anosov representation $\rho:\Gamma\to\g$, the set $\bfo_\rho^\bfi$ is a domain of discontinuity for $\rho$.
\end{thm}

Theorem \ref{thm: dod for anosov in intro} applies to a very general class of spaces $\x$.
This makes the notion of fat ideal hard to track.
We therefore refine the above result for homogeneous spaces $\x$ for which more structure is available.

Part of this structure is the existence of a Cartan involution $\tau$ of $\g$ so that $\tau(\h)=\h$.
As already mentioned, such an involution always exists when $\x$ is a \textit{symmetric space}, i.e., when $\h$ coincides with the fixed point set of an involutive automorphism $\sigma:\g\to\g$ (see Subsection \ref{subsec: cartan for symmetric}).
In Proposition \ref{prop: action of w0} we show that the assumption $\tau(\h)=\h$ implies the existence of an order preserving involution $w_0$ of the set of relative positions with interesting properties (see Section \ref{sec: relative positions}).
With this involution at hand, we introduce in Section \ref{sec: fat ideals} the notion of $w_0$-fat ideal: an ideal $\bfi\subset\g\backslash(\f_\theta\times\x)$ is $w_0$-\textit{fat} if for every minimal relative position $\bfp\notin\bfi$ one has $w_0\cdot\bfp\in\bfi$.
This is similar to the notion of fat ideal introduced by Kapovich-Leeb-Porti \cite{KLPdomains} when $\x=\f_{\theta'}$.
In their setting, there is also an involution $w_0$ of the set of relative positions, but it is order reversing instead (note that no Cartan involution fixes $\p_{\theta'}$).
Kapovich-Leeb-Porti defined an ideal $\bfi$ to be fat if for every $\bfp\notin\bfi$ one has $w_0\cdot \bfp\in\bfi$.

Equipped with our notion of $w_0$-fat ideal we prove the following analogue of Theorem \ref{thm: dod for anosov in intro} which notably applies to symmetric spaces.

\begin{thm}[See Theorem \ref{thm: dod for anosov}(2)]\label{thm: dod for anosov in intro II}
Let $\theta\subset\Delta$ be a non-empty subset of simple roots of $\g$ and $\x$ be a $\g$-homogeneous space so that the action $\g\curvearrowright (\f_\theta\times\x)$ has finitely many orbits.
Assume moreover that there is a Cartan involution $\tau$ of $\g$ so that $\tau(\h)=\h$ and let $\bfi\subset \g\backslash(\f_\theta\times\x)$ be a $w_0$-fat ideal.
Then for every $\theta$-Anosov representation $\rho:\Gamma\to\g$, the set $\bfo_\rho^\bfi$ is a domain of discontinuity for $\rho$.
\end{thm}

Both our ``fat ideals'' and ``$w_0$-fat ideals'' are generalizations of the fat ideals from \cite{KLPdomains}, each with different advantages: the definition of ``fat ideals'' requires nothing except the set of relative positions being finite.
On the other hand, the notion ``$w_0$-fat ideals'' only applies in a more special situation, but can give larger domains, sometimes even maximal ones (see below).
Example \ref{ex: fat ideals in flag case}, Lemma \ref{lem: fat is w0-fat}, and Example \ref{ex: fat ideals lines transverse to hyperplanes and full flags} compare these and Kapovich-Leeb-Porti's definition in more detail.

Another advantage of Theorem \ref{thm: dod for anosov in intro II} is that it is easier to apply than Theorem \ref{thm: dod for anosov in intro}.
For instance, it readily implies Corollary \ref{cor: dod for non closed positions in intro}, since the set of minimal positions is a $w_0$--fat ideal.
More generally, to understand $w_0$-fat ideals we only need to consider $w_0$-orbits of minimal positions, regardless of how the whole picture of the Bruhat order looks like (for a taste of how this works concretely, see Examples \ref{ex: fat ideals complementary subspaces and projective} and \ref{ex: fat ideals quadratic forms and projective}).
By following this principle, in Subsection \ref{subsec: exampples of dods} we discuss families of examples where Theorem \ref{thm: dod for anosov in intro II} applies.
For instance, we construct domains of discontinuity in group manifolds (Example \ref{ex: dod group manifolds}) and in spaces of quadratic forms of given signature on a real vector space (Example \ref{ex: dod quadratic forms projetive anosov}). We also construct domains in \begin{equation}\label{eq: complementary subspaces in intro}
\x:=\{(U^+,U^-)\in \mathscr{G}_p(\rr^d)\times\mathscr{G}_q(\rr^d): U^+\oplus U^-=\rr^d \},
\end{equation}
\noindent where $\mathscr{G}_i(\rr^d)$ denotes the Grassmannian of $i$-dimensional subspaces of $\rr^d$, and $p$ and $q$ are positive integers with $d=p+q$. 
See Example \ref{ex: dod complementary subspaces projective} for details.

Finally, we prove a maximality result.
Note that if $\bfi'\subset\bfi$ then $\bfo_\rho^\bfi\subset\bfo_\rho^{\bfi'}$, hence we need to care about minimal $w_0$-fat ideals (c.f. Figure \ref{fig: intro}).
In the case $\x=\f_{\theta'}$ and for $\Delta$-Anosov representations, maximal open domains of discontinuity correspond precisely to minimal fat ideals, as proven in \cite{SteIDEALS}.
Our final result is a statement in the same direction when $\x$ is a symmetric space satisfying an extra hypothesis.

As mentioned before, the action of a discrete subgroup of $\g$ on $\x$ may be proper itself.
Theorems \ref{thm: dod for anosov in intro} and \ref{thm: dod for anosov in intro II} are independent of this, but a maximality result must, somehow, take Benoist-Kobayashi's Theorem \ref{thm: benoist kobayashi} into account.
In our case we require that the interior of the \textit{limit cone} $\cone$ of $\rho$ intersects $\mu(\h)$.
Recall that the limit cone of $\rho$ is a fundamental object in the study of asymptotic properties of $\rho$.
It was introduced by Benoist \cite{BenPAGL} who also showed some remarkable properties, like that it is convex with non-empty interior when $\rho$ is Zariski dense (see Subsection \ref{subsec: limit cone} for details).

In the following theorem we assume that $\h$ is symmetric and $\tau$ is a Cartan involution with $\tau(\h) = \h$.
In this setting we can find a Cartan subspace $\liea $ of $\g$ whose restriction to the Lie algebra $\lieh$ of $\h$ is a Cartan subspace of $\h$.
We let $\ko$ be the maximal compact subgroup associated to $\tau$ and $\m$ the centralizer in $\ko$ of $\liea$.

\begin{thm}[See Theorem \ref{thm: maximality with w0-fat}]\label{thm: maximality in intro}
Suppose that $\h$ is symmetric, $\liea$ is $\sigma$-invariant, and $\liea_\h:=\liea\cap\lieh$ is a Cartan subspace of $\h$.
Assume also $\m\subset\h$.
Let $\rho:\Gamma\to\g$ be a Zariski dense $\Delta$-Anosov representation so that $\mu(\h)\cap\tn{int}(\cone)\neq\emptyset$.
Then if $\bfo\subset \x$ is a maximal open domain of discontinuity for $\rho$, there exists a $w_0$-fat ideal $\bfi\subset\g\backslash(\f_\Delta\times\x)$ so that $\bfo=\bfo_\rho^\bfi$.
\end{thm}

A family of examples in which Theorem \ref{thm: maximality in intro} applies is given by the symmetric spaces (\ref{eq: complementary subspaces in intro}) of complementary subspaces in $\rr^d$, see Example \ref{ex: maximal complementary subspaces}. 
See also Remark \ref{rem: MsubsetH and oriented} and Example \ref{ex: MsubsetH is needed for dynamical relation} for comments on the assumption $\m\subset\h$ in Theorem \ref{thm: maximality in intro}.

\subsection{Outline of the proof}\label{subsec: outline of proof intro}

We now discuss informally the main ideas behind the proofs of Theorems \ref{thm: dod for anosov in intro}, \ref{thm: dod for anosov in intro II} and \ref{thm: maximality in intro}.
Along the way, we provide intuitions behind the notions of fat and $w_0$-fat ideals.

A central property of $\theta$-Anosov representations, already used by Guichard-Wien\-hard \cite{GW} and Kapovich-Leeb-Porti \cite{KLPdomains} in their construction of domains of discontinuity, is the fact that the $\Gamma$-action on $\Lambda_\rho^\theta$ is ``uniformly hyperbolic".
This means that for every sequence $\gamma_n\to\infty$ in $\Gamma$ one may find flags $\xi_+$ and $\xi_-$ in $\Lambda_\rho^\theta$ with the property that, as $n \to \infty$,
\begin{equation}\label{eq: uniformly hyperoblic dynamics intro}
  \rho(\gamma_n)|_{C(\xi_-)} \rightarrow \xi_+
\end{equation}
locally uniformly as maps on the flag manifold.
Here $C(\xi_-)$ is the open set of flags transverse to $\xi_-$.

On the other hand, if the $\Gamma$-action on $\bfo_\rho^\bfi$ is not proper, there must be points $x,x'\in\bfo_\rho^\bfi$ which are \textit{dynamically related} under a sequence in $\rho(\Gamma)$, i.e. there exist sequences $x_n\to x$ and $\gamma_n\to\infty$ so that
\begin{equation}\label{eq: dyn rel intro}
\rho(\gamma_n)\cdot x_n\to x'.
\end{equation}
A key step in our argument is to combine the uniformly hyperbolic nature of the dynamical system $\Gamma\curvearrowright\f_\theta$ with this remark.
Indeed, Equations (\ref{eq: uniformly hyperoblic dynamics intro}) and (\ref{eq: dyn rel intro}) express that some relative positions are ``degenerating" to a lower position.
Consequently, for every $\xi\in\f_\theta$ transverse to $\xi_-$ we must have
\[\pos(\xi_+,x')\leq \pos(\xi,x).\]
\noindent This is the content of Lemma \ref{lem: for domains of discontinuity}, which is adapted from the flag case \cite[Proposition 6.2]{KLPdomains}.

The above argument suggests to introduce a symmetric relation $\leftrightarrow$ on the set of relative positions.
If $\bfp,\bfp'\in\g\backslash(\f_\theta\times\x)$, we say that $\bfp$ is \textit{transversely related} to $\bfp'$, and denote $\bfp\leftrightarrow\bfp'$, if there exist two transverse flags $\xi_1,\xi_2\in\f_\theta$ and a point $x\in\x$ so that
\[\bfp=\pos(\xi_1,x) \tn{ and } \bfp'=\pos(\xi_2,x).\]
In Section \ref{sec: fat ideals} we define an ideal to be \textit{fat} if for every $\bfp\notin\bfi$ there exists $\bfp'\in\bfi$ so that $\bfp'\leftrightarrow\bfp$.
When $\x=\f_{\theta'}$, this precisely coincides with the notion of fat ideal introduced by Kapovich-Leeb-Porti \cite{KLPdomains} (see Example \ref{ex: fat ideals in flag case}).
Equipped with this definition, Lemma \ref{lem: for domains of discontinuity} readily implies Theorem \ref{thm: dod for anosov in intro} (see Theorem \ref{thm: dod for anosov}(1)).

The proof of Theorem \ref{thm: dod for anosov in intro II} is also crucially based on Lemma \ref{lem: for domains of discontinuity}.
Indeed, suppose that $\tau(\h)=\h$.
We then have an order preserving involution $w_0$ of the set of relative positions (Proposition \ref{prop: action of w0}).
Further, in Corollary \ref{cor: w0 action and transversely related} we prove that for every relative position $\bfp$ one has $\bfp\leftrightarrow w_0\cdot \bfp$.
Thus, the notion of $w_0$-fat ideal is designed precisely to apply Lemma \ref{lem: for domains of discontinuity} and obtain a contradiction if we suppose that the $\Gamma$-action on $\bfo_\rho^\bfi$ is not proper (see Theorem \ref{thm: dod for anosov}(2) for details).

Finally, we comment on the proof of Theorem \ref{thm: maximality in intro}.
The key step is Proposition \ref{prop suf condition for dynamical relation symmetric case}, which can be thought of as a converse of Lemma \ref{lem: for domains of discontinuity} in the sense that gives a sufficient condition for having a dynamical relation between given points in $\x$, in terms of their relative positions with points in the limit set $\Lambda_\rho^\Delta$.
Once this is proved, the maximality result follows by a general argument.

The proof of Proposition \ref{prop suf condition for dynamical relation symmetric case} is based on two ingredients.
One is a description of the set of minimal relative positions when $\x$ is symmetric and $\theta=\Delta$, due to Matsuki \cite{MatsukiMinimalParabolics} (see Theorem \ref{thm: matsuki minimal}).
This result states that minimal positions are parametrized by the subset $\w^\sigma$ of the Weyl group that preserves $\liea_\h=\liea\cap\lieh$.
In Corollary \ref{cor: related minimal positions} we apply this result to show that the transverse relation $\bfp \leftrightarrow \bfp'$ is equivalent to $\bfp' = w_0 \cdot\bfp$ for minimal positions $\bfp$ and $\bfp'$.
This gives us a nice relation between the relevant relative positions $\pos(\xi_+,x')$ and $\pos(\xi_-,x)$, and a very concrete way of representing them by $w$ and $w_0w$ respectively, for some $w\in \w^\sigma$.

The other ingredient is that, by definition of $\w^\sigma$, for every $H\in\liea\cap\lieh$ one has \begin{equation}\label{eq:  action of H on wo intro}
\exp(H)w\cdot o=w\cdot o.
\end{equation}
\noindent Hence, for sequences $\{\rho(\gamma_n)\}$ for which $\mu(\rho(\gamma_n))$ is close to $\mu(\h)$ we get good control on the displacement of the point $w\cdot o\in\x$.
This is captured by Proposition \ref{prop: directions interior to limit cone approached by mu(Gamma)}, where we apply deep results by Benoist \cite{BenPAGL,BenPAGLII} to find a sequence $\gamma_n\to\infty$ whose attracting and repelling points approach $\xi_+$ and $\xi_-$ respectively, and for which the Cartan projection $\mu(\rho(\gamma_n))$ approaches a given direction in $\mu(\h)\cap\tn{int}(\cone)$.
This allows us to conclude with essentially the same argument as in \cite{SteIDEALS}.

\subsection{Final remarks and future directions}

Recall that Guichard-Wienhard \cite{GW} and Kapovich-Leeb-Porti \cite{KLPdomains} constructed domains of discontinuity in $\x=\f_{\theta'}$ which are moreover co-compact.
It is natural to ask whether we can generalize this result to construct co-compact domains of discontinuity in our setting.
Observe however that in the setting of \cite{GW,KLPdomains}, a crucial ingredient to prove co-compactness is the fact that $\x=\f_{\theta'}$ is compact itself, a feature that no longer holds in our more general setting.
Note that even if our domains do not give compact quotients, it might be possible to compactify them using an approach like Gu\'eritaud-Guichard-Kassel-Wienhard \cite{GGKWtame}.


As another direction, it would be interesting to understand the topology of the quotients $\rho(\Gamma)\backslash\bfo_\rho^\bfi$.
In the case that $\x$ is a flag manifold, Guichard-Wienhard \cite{GW} show that the topological type of $\rho(\Gamma)\backslash\bfo_\rho^\bfi$ stays the same when $\rho$ is deformed continuously.
Due to the lack of compactness, an analogue of this theorem for general $\x$ would require a different argument, and understanding the topology of the quotient might be difficult.
A special case is that of $\mathbb{A}\tn{d}\mathbb{S}$-quasi-Fuchsian $3$-manifolds studied by Mess \cite{Mess}, in which the quotient space is always homeomorphic to the product of a closed surface with the real line.


\subsection{Organization of the paper}

The paper is structured as follows.
In Section \ref{sec: cartan for symmetric} we discuss well known preliminaries and fix some notations about the structure theory of semi-simple Lie groups and their symmetric spaces. 
In Sections \ref{sec: relative positions} and \ref{sec: fat ideals} we develop the formalism of relative positions and fat ideals. 
In Section \ref{sec: anosov} we recall Benoist's results on the limit cone and introduce Anosov representations. 
We prove Theorems \ref{thm: dod for anosov in intro} and \ref{thm: dod for anosov in intro II} in Section \ref{sec: dod}, where we also include a detailed discussion of examples. 
Finally, we prove the maximality Theorem \ref{thm: maximality in intro} in Section \ref{sec: maximality}.

Dependence between sections is as follows (in particular, Sections \ref{subsec: cartan for symmetric}, \ref{subsec: minimal transversely related} and \ref{subsec: limit cone} are only needed for the maximality theorem in Section \ref{sec: maximality}):

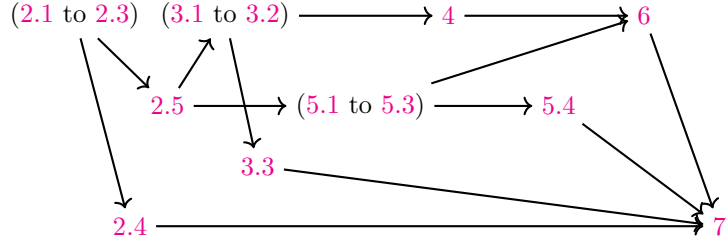
\begin{center}
\begin{tikzpicture}
\node[left] (A) at (-5.8,-3) {(\ref{subsec: roots} to 
\ref{subsec: weyl})};

\node[left] (B) at (-5.7,-5.8) { \ref{subsec: cartan for symmetric}};

\node[left] (C) at (-5.2,-4.2) { \ref{subsec: flags}};

\node[left] (D) at (-3.8,-3) {(\ref{subsec: relative positions and order} to \ref{subsec: transversely related})};
\node[left] (E) at (-1.6,-3) {\ref{sec: fat ideals}};

\node[left] (F) at (1,-3) {\ref{sec: dod}};
\node[left] (G) at (-2,-4.2) {(\ref{subsec: tits} to \ref{subsec: anosov reps})};
\node[left] (H) at (2,-5.8) {\ref{sec: maximality}};
\node[left] (I) at (-4,-5) {\ref{subsec: minimal transversely related}};
\node[left] (J) at (0,-4.2) {\ref{subsec: limit cone}};

\draw[->,thick] (A) -- (B);
\draw[->,thick] (A) -- (C);
\draw[->,thick] (B) -- (H);
\draw[->,thick] (C) -- (D);
\draw[->,thick] (C) -- (G);
\draw[->,thick] (D) -- (E);
\draw[->,thick] (E) -- (F);
\draw[->,thick] (F) -- (H);
\draw[->,thick] (G) -- (F);
\draw[->,thick] (I) -- (H);
\draw[->,thick] (D) -- (I);
\draw[->,thick] (G) -- (J);
\draw[->,thick] (J) -- (H);
\end{tikzpicture}
\end{center}

\subsection{Acknowledgements}

We are grateful to Jeffrey Danciger, Rafael Potrie, Sara Maloni, Beatrice Pozzetti, Andr\'es Sambarino and Anna Wienhard for many helpful discussions.
We acknowledge the warm hospitality of Universit\"at Heidelberg where large parts of this work were completed.

\section{Cartan decomposition for symmetric spaces}\label{sec: cartan for symmetric}

We begin by fixing notations and recalling part of the structure theory of semisimple Lie groups and symmetric spaces.
All the material covered in this section is standard (expect possibly for that of Subsection \ref{subsec: cartan for symmetric}), and the reader is referred to Knapp \cite{Kna} and Schlichtkrull \cite[Chapter 7]{Sch} for further details.

Throughout the paper we fix a linear, connected, finite center, semisimple Lie group $\g$ without compact factors.
We also let $\x$ be a $\g$-homogeneous space.
We fix a basepoint $o\in\x$ and let $\h<\g$ be its stabilizer in $\g$.
Hence $\x\cong\g/\h$.

\subsection{Roots, Weyl chambers and Cartan projection}\label{subsec: roots}

Let $\lieg$ be the Lie algebra of $\g$ and $\kappa$ be its Killing form.
A \textit{Cartan involution} of $\lieg$ is an involutive automorphism $\tau:\lieg\to\lieg$ so that the bilinear form $$(X,Y)\mapsto -\kappa(X,\tau (Y))$$
\noindent is positive definite.
We let $\liek$ (resp. $\liep$) be the eigenspace of $\tau$ associated to the eigenvalue $1$ (resp. $-1$).

Fix a \textit{Cartan subspace} $\liea\subset\liep$, i.e. a maximal (abelian) subalgebra contained in $\liep$.
Let $\Sigma\subset\liea^*$ be the set of \textit{restricted roots} of $\liea$ on $\lieg$.
By definition, a non-zero $\alpha\in\liea^*$ belongs to $\Sigma$ if and only if the \textit{root space}$$\lieg_\alpha:=\{X\in\lieg: [A,X]=\alpha(A)X \tn{ for all } A\in\liea\}$$
\noindent is non-zero.
An element of $\liea$ is \textit{regular} if it belongs to $\liea\setminus(\cup_{\alpha\in\Sigma}\ker\alpha)$.

Let $\liea^+\subset\liea$ be a \textit{Weyl chamber}, i.e. the closure of a connected component of $\liea\setminus(\cup_{\alpha\in\Sigma}\ker\alpha)$.
Its interior is denoted by $\tn{int}(\liea^+)$.
Let $\Sigma^+:=\{\alpha\in\Sigma: \alpha\vert_{\liea^+}\geq 0\}$ be the corresponding positive system, and $\Delta\subset\Sigma^+$ be the set of simple roots.
Then $\Delta$ is a basis of $\liea^*$ on which the coefficients of every positive root $\alpha\in\Sigma^+$ are non-negative. See Example \ref{ex: flags sld} for a concrete example.

Let $\ko<\g$ be the connected Lie subgroup associated to the Lie algebra $\liek$.
It is a maximal compact subgroup of $\g$.
We have the \textit{Cartan decomposition} $\g=\ko\exp(\liea^+)\ko$ of $\g$ and a corresponding \textit{Cartan projection} $\mu:\g\to\liea^+$ characterized, for all $g\in\g$, by $$g\in\ko\exp(\mu(g))\ko.$$

\subsection{Symmetric subgroups}\label{subsec: symmetric subgroups}

Let $\lieh\subset\lieg$ be the Lie algebra of $\h$.
A special case of interest is the following: the group $\h$ is \textit{symmetric} if it coincides with the fixed point set of an involutive automorphism $\sigma:\g\to\g$.
In this case we say that $\x$ is a \textit{symmetric space}.
We will also denote by $\sigma$ the induced Lie algebra automorphism.
In particular, $\lieh$ coincides with the fixed point set of $\sigma:\lieg\to\lieg$.

The Cartan involution $\tau$ may be chosen in such a way that $\sigma\tau=\tau\sigma$ (see Schlichtkrull \cite[Proposition 7.1.1]{Sch}), and we will assume this choice whenever working with a symmetric $\h$.
Note that $\tau$ preserves the splitting $\lieg=\lieh\oplus\lieq$, where $$\lieq:=\{X\in\lieg: \sigma(X)=-X\}.$$
\noindent The group $\h$ is reductive, $\ko\cap\h$ is a maximal compact subgroup of $\h$ and $$\lieh=(\lieh\cap\liek)\oplus(\lieh\cap\liep)$$ \noindent is a Cartan decomposition of $\lieh$ (see \cite[p.115]{Sch}).

In addition, the Cartan subspace $\liea$ may be chosen to be $\sigma$-invariant and we always assume this is the case when working in the symmetric setting.
Further, we will always pick our Cartan subspace $\liea$ so that $\liea_\h:=\liea\cap\lieh$ is a Cartan subspace of $\lieh$.
The dual space $(\liea_\h)^*$ is naturally embedded in $\liea^*$ by extending a functional $\liea_\h\to\rr$ to be zero on $\liea\cap\lieq$.
Let $\Sigma_\h$ be the set of restricted roots of $\liea_\h$ in $\lieh$, that is, the set of non-zero functionals in $(\liea_\h)^*$ for which the root space $$\lieh_\alpha:=\{X\in\lieh: [A,X]=\alpha(A)X \tn{ for all } A\in\liea_\h\}$$ \noindent is non-zero.

For $\alpha\in\Sigma$ we denote $\sigma(\alpha):=\alpha\circ\sigma$.
Note that $\lieg_{\sigma(\alpha)}=\sigma(\lieg_\alpha)$ and, as a consequence, $\sigma$ induces an involution of $\Sigma$.
The Weyl chamber $\liea^+$ (or equivalently $\Sigma^+$) is said to be $\sigma$-\textit{compatible} if for every $\alpha\in\Sigma^+$ with $\alpha\vert_{\liea_\h}\neq 0$ one has $\sigma(\alpha)\in\Sigma^+$.
We record the following remark for future use.

\begin{rem}\label{rem: sigma compatible and regular}
Suppose that the Weyl chamber $\liea^+$ satisfies $\liea_\h\cap\tn{int}(\liea^+)\neq\emptyset$.
Then $\liea^+$ is $\sigma$-compatible.
\end{rem}

\subsection{Weyl groups}\label{subsec: weyl}

We let $\n_\ko(\liea)$ (resp. $\m:=\z_\ko(\liea)$) be the normalizer (resp. centralizer) of $\liea$ in $\ko$.
The \textit{Weyl group} of $\Sigma$ is $\w:=\n_\ko(\liea)/\m$.
It acts simply transitively on the set of Weyl chambers of $\Sigma$.

For $\alpha\in\Delta$ we let $s_\alpha\in\w$ be the corresponding \textit{root reflection}, i.e. the element acting on $\liea$ as an orthogonal reflection, with respect to the Killing form, along $\ker\alpha$.
The set $\{s_\alpha\}_{\alpha\in\Delta}$ generates $\w$.
With respect to this generating set, there exists a unique longest element $w_0\in\w$.
We let $\iota:=-w_0:\liea\to\liea$ be the corresponding \textit{opposition involution}.
Note that $\iota$ acts on $\Delta$ and therefore $\iota(\liea^+)=\liea^+$.

In the symmetric setting, two other ``Weyl groups" will be of interest.
Firstly, we may consider the group $$(\n_{\ko}(\liea)\cap\h)/(\m\cap\h).$$ \noindent It naturally embeds as a subgroup of $\w$.
Abusing notations, we will denote this subgroup by $\w\cap\h\subset\w$.
Secondly, note that both $\n_\ko(\liea)$ and $\m$ are $\sigma$-invariant, so $\sigma$ induces an involution of $\w$. We denote its fixed point set by $\w^\sigma$.

In other words, if $\widetilde w \in \n_\ko(\liea)$ is a representative of $w \in \w$, then $w \in \w^\sigma$ means that $\sigma(\widetilde w)$ and $\widetilde w$ represent the same element in $\w$, while $w \in \w \cap \h$ means that $\widetilde w$ can be chosen so that $\sigma(\widetilde w) = \widetilde w$.
Observe that $\w \cap \h$ is a normal subgroup of $\w^\sigma$.
In general, we have $\w\cap\h\neq\w^\sigma$ (see Example \ref{ex: matsuki and w0 action complementary subspaces}).



\subsection{Cartan projection of symmetric subgroups}\label{subsec: cartan for symmetric}

Suppose that $\h$ is symmetric and let $\liea$ be a $\sigma$-invariant Cartan subspace of $\lieg$ so that $\liea_\h:=\liea\cap\lieh$ is a Cartan subspace of $\lieh$.
Assuming that $\liea_\h$ contains a regular element, we will now prove that the Cartan projection $\mu$ can be chosen in such a way that $\mu(\h)\subset\lieh$ (Corollary \ref{cor cartan compatible with H} below).
We include a proof as we were unable to find a reference for it, although this fact is likely well--known.
It will be used for the proof of Proposition \ref{prop suf condition for dynamical relation symmetric case}.





We first need a preparatory lemma, which describes the root system of $\liea_\h$ acting on $\lieh$ in terms of the root system of $\liea$ acting on $\lieg$.

\begin{lem}\label{lema roots of liea cap lieh and liea}
Suppose that $\h$ is symmetric and that $\liea_\h=\liea\cap\lieh$ is a Cartan subspace of $\lieh$.
Then the following equality holds:
\[\Sigma_\h = \lbrace \alpha|_{\liea_\h} \colon \ \alpha\in\Sigma \tn{ and } \lieg_{\alpha}\not\subset\lieq\rbrace\setminus\lbrace 0 \rbrace.\]
\end{lem}

\begin{proof}

Let $\alpha\in\Sigma$ be such that $\alpha\vert_{\liea_\h}\neq 0$ and $\lieg_{\alpha}\not\subset\lieq$.
There exists then $X_0\in\lieg_{\alpha}$ such that $X:=X_0+\sigma(X_0)\neq 0$.
Then $X \in \lieh$ and $X$ is a root vector for $\alpha|_{\liea_\h}$, as
\[[A,X] = [A,X_0] + [A,\sigma(X_0)] = [A,X_0] + \sigma[A,X_0] = \alpha(A)(X_0 + \sigma(X_0)) = \alpha(A)X\]
for all $A \in \liea_\h \subset \liea$.

Conversely, let $\alpha\in\Sigma_{\h}$ and $X\in\lieh_\alpha$ be a non-zero vector.
We may decompose $X$ as
\[X=X_0+\displaystyle\sum_{\widehat{\alpha}\in\Sigma}X_{\widehat{\alpha}},\]
with $X_0 \in \lieg_0$ and $X_{\widehat\alpha} \in \lieg_{\widehat\alpha}$.
It follows that for every $A\in\liea_\h$,
\[\alpha(A)X=\displaystyle\sum_{\widehat{\alpha}\in\Sigma}\widehat{\alpha}(A)X_{\widehat{\alpha}}.\]
Hence $X_0=0$ and $\alpha(A)X_{\widehat{\alpha}}=\widehat{\alpha}(A)X_{\widehat{\alpha}}$ for every $\widehat{\alpha}\in\Sigma$.
Since $X\neq 0$ there is $\widehat\alpha \in \Sigma$ with $\alpha = \widehat\alpha|_{\liea_\h}$.
Furthermore, $\widehat\alpha$ can be chosen with $\lieg_{\widehat\alpha} \not\subset \lieq$, as otherwise $X \in \lieq$, which would contradict $X \in \lieh$ and $X \neq 0$.
\end{proof}

\begin{cor}\label{cor cartan compatible with H}
Suppose that $\h$ is symmetric, $\liea$ is $\sigma$-invariant, and $\liea_\h:=\liea\cap\lieh$ is a Cartan subspace of $\h$.
Assume that $\liea_\h$ contains a regular element, and pick the Weyl chamber $\liea^+$ in such a way that $\tn{int}(\liea^+)\cap\liea_\h$ is non-empty.
Then the Cartan projection $\mu:\g\to\liea^+$ satisfies $\mu(h)\in\liea_\h$ for all $h\in\h$.
\end{cor}

\begin{proof}

As $\liea_\h$ contains a regular element of $\liea$, for every $\alpha\in\Sigma$ we have $\alpha\vert_{\liea_\h}\neq 0$.
Lemma \ref{lema roots of liea cap lieh and liea} then implies
$$ \Sigma_\h=\lbrace  \widehat{\alpha}|_{\liea_\h} \colon \ \widehat{\alpha}\in\Sigma \tn{ and } \lieg_{\widehat{\alpha}}\not\subset\lieq\rbrace.$$
\noindent Let $\Sigma_\h^+$ be a positive system of $\Sigma_\h$ and $\liea_\h^+\subset\liea_\h$ be the corresponding Weyl chamber.
Further, we pick $\liea_\h^+$ in such a way that it intersects $\tn{int}(\liea^+)$.
Moreover, the set $\liea_\h^+\setminus\left(\cup_{\widehat{\alpha}\in\Sigma}\ker\widehat{\alpha}\right)$ is non-empty and has a finite number of connected components, consisting only of regular elements of $\liea$.
Hence, to each of these connected components corresponds a unique Weyl chamber of $\Sigma$, one of which is $\liea^+$.
We denote these Weyl chambers by $\liea^+=\liea_1^+,\dots,\liea_k^+$.
By Remark \ref{rem: sigma compatible and regular}, the positive system associated to each of the $\liea_i^+$ is $\sigma$-compatible.
Applying Schlichtkrull \cite[Lemma 7.1.6]{Sch}, we conclude that for every $i=1,\dots,k$, the element $w_i\in\w$ taking $\liea_i^+$ to $\liea^+$ belongs to $\w^\sigma$.

On the other hand, we have the Cartan decomposition $$\h=(\ko\cap\h)\exp(\liea_\h^+)(\ko\cap\h),$$ \noindent and a corresponding Cartan projection $\mu_\h:\h\to\liea_\h^+$.
Let $h\in\h$ and write $h=k\exp(\mu_\h(h))l$, with $k,l\in\ko\cap\h$.
There is some $i=1,\dots,k$ so that $\mu_\h(h)\in\liea_i^+$ and therefore $w_i\cdot \mu_\h(h)\in\liea^+$.
As $w_i$ lifts to an element in $\ko$, we have $w_i\cdot \mu_\h(h) =\mu(h)$.
Since $w_i\in\w^\sigma$, the proof is finished.
\end{proof}




\subsection{Flag manifolds}\label{subsec: flags}

Let $\theta\subset\Delta$ be a non-empty subset.
We let
\[\liep_\theta \coloneqq \!\!\! \bigoplus_{\alpha\in\Sigma^+ \cup \langle \Delta \setminus \theta \rangle} \!\!\! \lieg_\alpha, \qquad \liep_\theta^- \coloneqq \!\!\! \bigoplus_{\alpha\in\Sigma^+ \cup \langle \Delta \setminus \theta \rangle} \!\!\! \lieg_{-\alpha},\]
where $\langle\Delta\setminus\theta\rangle$ denotes the set of roots (including 0) in the span of $\Delta\setminus\theta$.
We let $\p_\theta$ and $\p_\theta^-$ be the corresponding (opposite) parabolic subgroups of $\g$.
Observe that $\p_\theta^-$ is conjugate to $\p_{\iota(\theta)}$.
We say that $\p_\theta$ is \textit{self-opposite} if $\iota(\theta)=\theta$.


\begin{rem}\label{rem: action of tau and w0 on lieptheta}
Suppose that $\theta$ satisfies $\iota(\theta)=\theta$.
Then (for any lift of $w_0$ to $\n_\ko(\liea)$) one has $\tau(\liep_\theta)=\tn{Ad}_{w_0}(\liep_\theta)$, where $\tn{Ad}:\g\to\tn{Aut}(\lieg)$ is the adjoint representation.
In particular, $\tau(\p_\theta)=w_0\p_\theta w_0$.
\end{rem}

The \textit{flag manifolds} are the $\g$--homogeneous spaces
\[\f_\theta:=\g/\p_\theta \tn{ and } \f_\theta^-:=\g/\p_\theta^-,\]
\noindent which are also $\ko$--homogeneous thanks to the Iwasawa Decomposition Theorem.
For $\theta=\Delta$ we let $\f:=\f_\Delta\cong\f_\Delta^-$, $\bor:=\p_\Delta$ and $\bor^-:=\p_\Delta^-$.

The $\g$-orbit of $(\p_\theta^-,\p_\theta)$ is the unique open orbit of the action $\g\curvearrowright\f_\theta^-\times\f_\theta$.
We say that $\xi_-\in\f_\theta^-$ is \textit{transverse} to $\xi_+\in\f_\theta$ if $(\xi_-,\xi_+)$ belongs to this open orbit.


\begin{ex}\label{ex: flags sld}
Let $\g=\PSL(V)$, where $V$ is a real (resp. complex) vector space of dimension $d\geq 2$.
The Lie algebra of $\g$ is the space of traceless linear operators on $V$.
A maximal compact subgroup $\ko$ is the subgroup of orthogonal (resp. unitary) matrices with respect to an inner (resp. Hermitian inner) product $\langle\cdot,\cdot\rangle$ on $V$.
The choice of a Cartan subspace $\liea \subset \liep$ and a Weyl chamber $\liea^+ \subset \liea$ corresponds to the choice of an ordered orthonormal basis $(e_1,\dots,e_d)$ for $V$.
Then $\liea$ is the subalgebra of diagonal matrices with respect to this basis, and $\liea^+$ consists of diagonal matrices whose diagonal entries are distinct and in descending order.

If $\lambda_j(A)$ denotes the eigenvalue of $A\in\liea$ in direction $e_j$, and $\alpha_{ij} \coloneqq \lambda_i-\lambda_j$ for $i \neq j$, then
\[\Sigma=\{\alpha_{ij} \colon i\neq j\}, \qquad \Sigma^+=\{\alpha_{ij} \colon i< j\}, \qquad \Delta=\{\alpha_{i,i+1} \colon 1 \leq i \leq d -1\}.\]
Sometimes we will write the elements of $\Delta$ by $\alpha_i \coloneqq \alpha_{i,i+1}$.
The opposition involution maps $\lambda_j$ to $-\lambda_{d+1-j}$ and accordingly $\alpha_j$ to $\alpha_{d-j}$.

The choice of $\theta$ corresponds to the choice of a subset $\{i_1, \dots, i_k\} \subset \{1, \dots, d-1\}$, where $i_1 < \cdots < i_k$.
Then $\f_\theta$ identifies with the space of \textit{partial flags} indexed by $\theta$, that is, the space of sequences of the form $\xi_+=(\xi_+^{i_1}\subset\dots\subset\xi_+^{i_k})$, where $\xi_+^{i_j}$ is a linear subspace of $V$ of dimension $i_j$.
A flag $\xi_-\in\f_\theta^-$ is transverse to $\xi_+$ if and only if the sum $\xi_-^{d-i_j}+\xi_+^{i_j}$ is direct for all $j=1,\dots,k$.
When $\theta=\{k\}$, we use the special notation $\grass_k(V):=\f_\theta$ for the corresponding flag manifold, which is the \textit{Grassmannian} of $k$-dimensional subspaces of $V$.
If $k = 1$, it coincides with the projective space $\mathbb{P}(V)$.
\end{ex}

\section{Relative positions}\label{sec: relative positions}

Fix a non-empty subset $\theta\subset\Delta$ so that $\iota(\theta)=\theta$.
Throughout this section we take $\x = \g/\h$ to be any $\g$--homogeneous space for which the action $\g\curvearrowright\f_\theta\times\x$ has finitely many orbits.
For instance, this is always the case when $\h$ is a parabolic subgroup of $\g$ (see Knapp \cite[Theorem 7.40]{Kna}), or when $\h$ is symmetric (see Wolf \cite{Wol}).

\subsection{Relative positions and Bruhat order}\label{subsec: relative positions and order}

The $\g$-orbit of a point $(\xi,x)$ in $\f_\theta\times \x$ is called the \textit{relative position} between $\xi$ and $x$, and is denoted by $\pos(\xi,x)$.
The map
\[\pos(g_1\p_\theta,g_2\h)\mapsto\p_\theta g_1^{-1}g_2\h\]
induces an identification between the set of relative positions and the double quotient $\p_\theta\backslash\g/\h$.
In particular, the set of relative positions identifies with the set of $\p_\theta$-orbits on $\x = \g/\h$ and also with the set of $\h$-orbits on $\f_\theta = \g/\p_\theta$.
We will use these identifications throughout the paper.

The set $\p_\theta\backslash\g/\h$ carries a natural partial order $\leq$, which intuitively arranges relative positions according to their degree of ``genericity".
More precisely, for a pair of relative positions $\bfp$ and $\bfp'$ we will write $\bfp\leq\bfp'$ if there exist sequences $\xi_p\to\xi$ in $\f_\theta$ and $x_p\to x$ in $\x$ so that
\[\pos(\xi_p,x_p)=\bfp' \tn{ and } \pos(\xi,x)=\bfp.\]
Equivalently, representing $\bfp=\p_\theta g\h$ and $\bfp'=\p_\theta g'\h$ with $g,g'\in\g$	 we have
\[\bfp \leq \bfp' \quad \Leftrightarrow \quad \p_\theta g\h \subset\overline{\p_\theta g'\h}.\]
\begin{lem}\label{lem: partial order}
Assume that $\p_\theta\backslash\g/\h$ is finite.
Then $\leq $ defines a partial order on $\p_\theta\backslash\g/\h$.
\end{lem}

\begin{proof}
  The only thing to show is that $\overline{\p_\theta g \h} = \overline{\p_\theta g' \h}$ implies $\p_\theta g \h = \p_\theta g' \h$.
  We do this by proving that every orbit closure $\overline{\p_\theta g \h}$ contains a unique relatively open orbit, which is $\p_\theta g \h$.
  Note that every orbit $\p_\theta g \h \subset \g$ is an immersed submanifold and as such a countable union of embedded submanifolds. 
  Embedded submanifolds are always locally closed (i.e. the intersection of an open and a closed set).
  It is easy to see that a locally closed set is either nowhere dense or contains a non--empty open subset.

  By assumption $\overline{\p_\theta g \h}$ consists of finitely many orbits and hence countably many locally closed sets, which are also locally closed as subsets of $\overline{\p_\theta g \h}$.
  By the Baire Category Theorem, these cannot all be nowhere dense, so some double coset $\p_\theta g' \h \subset \overline{\p_\theta g \h}$ contains a subset which is open in $\overline{\p_\theta g \h}$.
  But then $\p_\theta g' \h$ must be open in $\overline{\p_\theta g \h}$, and in particular intersect the dense subset $\p_\theta g \h$.
  So $\p_\theta g' \h = \p_\theta g \h$, and this is the unique open orbit in $\p_\theta g \h$.
\end{proof}

The following lemma is a direct consequence of the above.

\begin{lem}\label{lem: maximal positions and openess}
A relative position $\bfp=\p_\theta g\h$ is maximal (resp. minimal) for the partial order $\leq$ if and only if the set $\p_\theta g\h$ is open (resp. closed) in $\g$.
Equivalently, for some (every) $x\in\x$ the set
\[\{ \xi\in\f_\theta: \pos(\xi,x)=\bfp \}\]
is open (resp. closed) in $\f_\theta$.
\end{lem}

Before continuing with the theory let us examine some examples.

\begin{ex}\label{ex: partial order when h parabolic}

The previous notions have been well studied when $\h$ equals a parabolic subgroup $\p_{\theta'}$, for some non-empty $\theta'\subset\Delta$.
Indeed, it follows from Bruhat Decomposition (see e.g. \cite[Chapter VII]{Kna}) that the set of relative positions $\p_\theta\backslash\g/\p_{\theta'}$ identifies with $$\w_{\theta,\theta'}:=\langle \Delta\setminus\theta\rangle \backslash \w /\langle \Delta\setminus\theta'\rangle,$$
\noindent where $\langle S\rangle \subset\w$ denotes the subgroup generated by the subset $S\subset\w$.
Concretely, one has $\langle \Delta\setminus \theta\rangle =\p_\theta\cap\w$.
The relative position associated to a Weyl group element $w$ is denoted by $[w]\in\w_{\theta,\theta'}$.

In this case the partial order between relative positions has a nice combinatorial description in terms of words in the generating set $\{s_\alpha \colon \alpha\in\Delta\}$.
If $w=s_{\alpha_1}\cdots s_{\alpha_k}$ is a reduced expression of a Weyl group element, then a relative position $[w']\in\w_{\theta,\theta'}$ is below $[w]$ if and only it can be represented as a sub-word of $s_{\alpha_1}\cdots s_{\alpha_k}$, i.e.
\[w'=s_{\alpha_1}^{i_1} \cdots s_{\alpha_k}^{i_k},\]
with $i_j\in\{0,1\}$ for all $j=1,\dots,k$.
In particular, there is a unique minimal position (represented by $1\in\w$), and a unique maximal one (represented by $w_0$). For more details, see e.g. \cite[Proposition 2.2.13]{SteThesis}. A particular case is shown in Figure \ref{fig: intro} left.
\end{ex}

\begin{ex}\label{ex: relative positions in group manifolds}
Let $\g_0$ be a linear, connected, finite center, semisimple Lie group without compact factors and with Lie algebra $\lieg_0$.
Consider $\g:=\g_0\times\g_0$ and let $\h<\g$ be the diagonal embedding of $\g_0$ into $\g$.
Then $\h$ equals the fixed point set of the involution
\[\sigma:\g\to\g: \sigma(g_{\tn{L}},g_{\tn{R}}):=(g_{\tn{R}},g_{\tn{L}}),\]
and the homogeneous space $\x \coloneqq \g/\h$ can be identified with $\g_0$. It is called the \emph{group manifold} associated to $\g_0$. Here $\g$ acts on $\g_0$ by
\[(g_{\tn{L}},g_{\tn{R}})\cdot g:=g_{\tn{L}}gg_{\tn{R}}^{-1}.\]

If $\tau_0$ is a Cartan involution of $\g_0$, then $\tau:=(\tau_0,\tau_0)$ is a Cartan involution of $\g$, which commutes with $\sigma$.
Let further $\liea_0$ be a Cartan subspace of $\g_0$, $\Sigma_0$ the set of roots of $\liea_0$ on $\lieg_0$, $\Sigma_0^+ \subset \Sigma_0$ a positive system and $\Delta_0 \subset \Sigma_0^+$ the set of simple roots.
Then $\liea \coloneqq \liea_0 \times \liea_0$ is a Cartan subspace of $\g$, the set of roots of $\liea$ on $\lieg$ is
\[\Sigma=\{(\alpha,0) \colon \alpha\in\Sigma_0\}\cup\{(0,\alpha) \colon \alpha\in\Sigma_0\}\]
(where $(\alpha,0)$ is the functional mapping $(A_{\tn{L}},A_{\tn{R}})$ to $\alpha(A_{\tn{L}})$ etc.) and
\[\Sigma^+ \coloneqq \{(\alpha,0) \colon \alpha\in\Sigma_0^+\}\cup\{(0,\alpha) \colon \alpha\in\Sigma_0^+\}\]
is a positive system.

Fix also two not simultaneously empty subsets $\theta_{\tn{L}}, \theta_{\tn{R}}\subset\Delta_0$, so that the corresponding flag manifolds $\f_{\theta_{\tn{L}}}$ and $\f_{\theta_{\tn{R}}}$ of $\g_0 $ are self-opposite.
Then $\f_{(\theta_{\tn{L}},\theta_{\tn{R}})}=\f_{\theta_{\tn{L}}}\times\f_{\theta_{\tn{R}}}$ is a self-opposite flag manifold of $\g$.
We have a one-to-one correspondence
\[\g\backslash( \f_{(\theta_{\tn{L}},\theta_{\tn{R}})} \times\x)\to\g_0\backslash(\f_{\theta_{\tn{L}}}\times\f_{\theta_{\tn{R}}}) \colon\quad \pos(\xi_{\tn{L}},\xi_{\tn{R}},g)\mapsto\pos(\xi_{\tn{L}},g\cdot\xi_{\tn{R}}).\]
It preserves the corresponding partial orders.
In particular, since there is a unique minimal position in $\g_0\backslash(\f_{\theta_{\tn{L}}}\times\f_{\theta_{\tn{R}}})$, the same holds in $\g\backslash( \f_{(\theta_{\tn{L}},\theta_{\tn{R}})} \times\x)$.


\end{ex}

\begin{ex}\label{ex: relative positions in hpq}
Let $2\leq p\leq q$ be two integers and fix a quadratic form of signature $(p,q)$ on $\rr^{p+q}$.
Then let $\g=\mathsf{PSO}_0(p,q)$ be the identity component the projectivized group of matrices preserving this quadratic form.

All flag manifolds of $\g$ are self-opposite, and they are parametrized by sequences $1\leq i_1<\dots< i_l\leq p$.
For such a sequence, the corresponding flag manifold is the space of sequences of the form
\[\xi^{i_1}\subset\dots\subset\xi^{i_l},\]
where $\xi^{i_j}$ is a totally isotropic subspace of $\rr^{p+q}$ of dimension $i_j$.

Let $\x$ be the \textit{pseudo-Riemannian hyperbolic space of signature $(p,q-1)$}, i.e, the space of negative lines for the underlying quadratic form.
It is a symmetric space, usually denoted by $\mathbb{H}^{p,q-1}$.
For $p=q=2$, there is a natural identification between this space and the group manifold $\PSL_2(\rr)$.

A natural boundary for $\mathbb{H}^{p,q-1}$ is $\partial\mathbb{H}^{p,q-1}$, the space of isotropic lines in $\rr^{p+q}$.
There are two relative positions in $\g \backslash(\partial\mathbb{H}^{p,q-1} \times \mathbb{H}^{p,q-1})$, as a point of $\mathbb{H}^{p,q-1}$ may or may not belong to the hyperplane $\xi^{\perp}$, for $\xi\in\partial\mathbb{H}^{p,q-1}$.
\end{ex}

\begin{ex}\label{ex: rel position pairs of complementary subspaces}

Let $\g = \SL(V)$ where $V$ is a real vector space of dimension $d\geq 2$.
Fix $1\leq p \leq q < d$ with $p + q = d$.
We consider the $\g$--space
\[\x \coloneqq \{(U^+,U^-)\in\grass_p(V)\times\grass_{q}(V) \mid U^+\oplus U^-=V  \}.\]
It is convenient to choose a basis $e_1,\dots,e_d$ of $V$ and write $U_o^+ = \langle e_1, \dots, e_{p} \rangle$ and $U_o^- = \langle e_{p+1}, \dots, e_d \rangle$.
Let $\h < \g$ be the stabilizer of the pair $(U_o^+,U_o^-)$ in $\x$.
It is the fixed point set of the involution $\sigma(g) = JgJ$ of $\g$, where $J$ is the diagonal matrix defined by $J|_{U^+_o} = 1$ and $J|_{U^-_o} = -1$.
In particular, $\x \cong \g/\h$ is a symmetric space.
In the basis $e_1, \dots, e_d$, the elements of $\h$ are the matrices which split into a $p \times p$ and a $q \times q$ block and have determinant one.
We also write $\h = \mathsf{S}(\GL_p(\rr) \times \GL_q(\rr))$.

The Cartan involution $\tau \colon \g \to \g$ given by the inverse transpose in the basis $e_1,\dots,e_d$ commutes with $\sigma$.
The Cartan subspace $\liea$ consisting of traceless diagonal matrices is contained in $\lieh$ and therefore $\sigma$ acts trivially on it.
We let $\theta=\Delta$, hence $\f=\f_\theta$ is the space of full flags in $V$ (recall Example \ref{ex: flags sld}).

The picture of the partial order for the case $p=1$ and $q=3$ is sketched in Figure \ref{fig: lines hyperplanes R4}.
There are four minimal positions, represented by $1, w_0,w_{13} ,w_0w_{13}$ (where $w_{ij}\in\w$ is the element that acts on $\{1,\dots ,d\}$ by transposing $i$ with $j$).
\end{ex}

\begin{figure}
  \begin{center}
  \begin{tikzpicture}[scale=0.7]
    \begin{scope}[shift={(0,0)}]
      \draw[thick,red] (0,0) -- (-1,0) -- (0,0.8);
      \draw[thick,blue] (0,0) -- (0,-1) -- (1,-0.2) -- (1,0.8);
      \draw[thick,red,fill=white] (0,0) -- (1,0) -- (2,0.8) -- (1,0.8);
      \draw[thick,blue] (0,0) -- (0,1) -- (1,1.8) -- (1,0.8);
      \draw[black,dashed] (0,0) -- (1,0.8);

      \fill[blue] (1.5,1.5) circle (0.1);
      \draw[thick,red] (-0.75,0.2) -- (0,0.2);
      \draw[thick,red] (0.25,0.2) -- (1.25,0.2);
      \fill[red] (0.75,0.2) circle (0.1);
    \end{scope}

    \begin{scope}[shift={(-2,-4)}]
      \draw[thick,red] (0,0) -- (-1,0) -- (0,0.8);
      \draw[thick,blue] (0,0) -- (0,-1) -- (1,-0.2) -- (1,0.8);
      \draw[thick,red,fill=white] (0,0) -- (1,0) -- (2,0.8) -- (1,0.8);
      \draw[thick,blue] (0,0) -- (0,1) -- (1,1.8) -- (1,0.8);
      \draw[black,dashed] (0,0) -- (1,0.8);

      \fill[blue] (1.25,0.6) circle (0.1);
      \draw[thick,red] (-0.75,0.2) -- (0,0.2);
      \draw[thick,red] (0.25,0.2) -- (1.25,0.2);
      \fill[red] (0.75,0.2) circle (0.1);
    \end{scope}

    \begin{scope}[shift={(2,-4)}]
      \draw[thick,red] (0,0) -- (-1,0) -- (0,0.8);
      \draw[thick,blue] (0,0) -- (0,-1) -- (1,-0.2) -- (1,0.8);
      \draw[thick,red,fill=white] (0,0) -- (1,0) -- (2,0.8) -- (1,0.8);
      \draw[thick,blue] (0,0) -- (0,1) -- (1,1.8) -- (1,0.8);
      \draw[black,dashed] (0,0) -- (1,0.8);

      \fill[blue] (1.5,1.5) circle (0.1);
      \draw[thick,red] (-0.75,0.2) -- (0,0.2);
      \draw[thick,red] (0.25,0.2) -- (1.25,0.2);
      \fill[red] (0.25,0.2) circle (0.1);
    \end{scope}

    \begin{scope}[shift={(-4,-8)}]
      \draw[thick,red] (0,0) -- (-1,0) -- (0,0.8);
      \draw[thick,blue] (0,0) -- (0,-1) -- (1,-0.2) -- (1,0.8);
      \draw[thick,red,fill=white] (0,0) -- (1,0) -- (2,0.8) -- (1,0.8);
      \draw[thick,blue] (0,0) -- (0,1) -- (1,1.8) -- (1,0.8);
      \draw[black,dashed] (0,0) -- (1,0.8);
\put (-120,-270) { $1$}

      \draw[thick,red] (0.5,0) -- (1.5,0.8);
      \fill[blue] (1.25,0.6) circle (0.1);
      \fill[red] (0.75,0.2) circle (0.1);
    \end{scope}
\put (-45,-270) { $w_0w_{13}$}

    \begin{scope}[shift={(0,-8)}]
      \draw[thick,red] (0,0) -- (-1,0) -- (0,0.8);
      \draw[thick,blue] (0,0) -- (0,-1) -- (1,-0.2) -- (1,0.8);
      \draw[thick,red,fill=white] (0,0) -- (1,0) -- (2,0.8) -- (1,0.8);
      \draw[thick,blue] (0,0) -- (0,1) -- (1,1.8) -- (1,0.8);
      \draw[black,dashed] (0,0) -- (1,0.8);
\put (40,-270) { $w_{13}$}

      \fill[blue] (1.25,0.6) circle (0.1);
      \draw[thick,red] (-0.75,0.2) -- (0,0.2);
      \draw[thick,red] (0.25,0.2) -- (1.25,0.2);
      \fill[red] (0.25,0.2) circle (0.1);
    \end{scope}
    \put (120,-270) { $w_0$}

    \begin{scope}[shift={(4,-8)}]
      \draw[thick,red] (0,0) -- (-1,0) -- (0,0.8);
      \draw[thick,blue] (0,0) -- (0,-1) -- (1,-0.2) -- (1,0.8);
      \draw[thick,red,fill=white] (0,0) -- (1,0) -- (2,0.8) -- (1,0.8);
      \draw[thick,blue] (0,0) -- (0,1) -- (1,1.8) -- (1,0.8);
      \draw[black,dashed] (0,0) -- (1,0.8);

      \fill[blue] (1.5,1.5) circle (0.1);
      \draw[thick,red] (0,0) -- (1,0.8);
      \fill[red] (0.25,0.2) circle (0.1);
    \end{scope}

    \begin{scope}[shift={(-6,-12)}]
      \draw[thick,red] (0,0) -- (-1,0) -- (0,0.8);
      \draw[thick,blue] (0,0) -- (0,-1) -- (1,-0.2) -- (1,0.8);
      \draw[thick,red,fill=white] (0,0) -- (1,0) -- (2,0.8) -- (1,0.8);
      \draw[thick,blue] (0,0) -- (0,1) -- (1,1.8) -- (1,0.8);
      \draw[black,dashed] (0,0) -- (1,0.8);

      \draw[thick,red] (0.5,0) -- (1.5,0.8);
      \fill[blue] (1.25,0.7) arc (90:270:0.1) -- cycle;
      \fill[red] (1.25,0.5) arc (-90:90:0.1) -- cycle;
    \end{scope}

    \begin{scope}[shift={(-2,-12)}]
      \draw[thick,red] (0,0) -- (-1,0) -- (0,0.8);
      \draw[thick,blue] (0,0) -- (0,-1) -- (1,-0.2) -- (1,0.8);
      \draw[thick,red,fill=white] (0,0) -- (1,0) -- (2,0.8) -- (1,0.8);
      \draw[thick,blue] (0,0) -- (0,1) -- (1,1.8) -- (1,0.8);
      \draw[black,dashed] (0,0) -- (1,0.8);

      \draw[thick,red] (-0.25,0) -- (0,0.1);
      \draw[thick,red] (0.25,0.2) -- (1.75,0.8);
      \fill[blue] (1.25,0.6) circle (0.1);
      \fill[red] (0.25,0.2) circle (0.1);
    \end{scope}

    \begin{scope}[shift={(2,-12)}]
      \draw[thick,red] (0,0) -- (-1,0) -- (0,0.8);
      \draw[thick,blue] (0,0) -- (0,-1) -- (1,-0.2) -- (1,0.8);
      \draw[thick,red,fill=white] (0,0) -- (1,0) -- (2,0.8) -- (1,0.8);
      \draw[thick,blue] (0,0) -- (0,1) -- (1,1.8) -- (1,0.8);
      \draw[black,dashed] (0,0) -- (1,0.8);

      \fill[blue] (1.25,0.6) circle (0.1);
      \draw[thick,red] (0,0) -- (1,0.8);
      \fill[red] (0.25,0.2) circle (0.1);
    \end{scope}

    \begin{scope}[shift={(6,-12)}]
      \draw[thick,red] (-1,0) -- (1,0) -- (2,0.8) -- (0,0.8) -- cycle;
      \draw[thick,blue] ($(-1,0)+(-0.135,-0.05)$) -- ($(1,0)+(0.01,-0.05)$) -- ($(2,0.8)+(0.135,0.05)$) -- ($(0,0.8)+(-0.01,0.05)$) -- cycle;

      \fill[blue] (1.5,1.5) circle (0.1);
      \draw[thick,red] (-0.75,0.2) -- (1.25,0.2);
      \fill[red] (0.75,0.2) circle (0.1);
    \end{scope}

    \draw[black,thick,->]  (1,-1.3) -- (2,-2.3);
    \draw[black,thick,->]  (0.2,-1.3) -- (-0.8,-2.3);

    \draw[black,thick,->]  (-2,-5.2) -- (-2.8,-6);
    \draw[black,thick,->]  (-1.1,-5.2) -- (-0.3,-6);
    \draw[black,thick,->]  (2,-5.2) -- (1,-6);
    \draw[black,thick,->]  (3,-5.2) -- (4,-6);

    \draw[black,thick,->]  (5,-9.3) -- (6,-10.3);
    \draw[black,thick,->]  (4,-9.3) -- (3.2,-10.3);
    \draw[black,thick,->]  (1,-9.3) -- (2,-10.3);
    \draw[black,thick,->]  (0,-9.3) -- (-0.8,-10.3);
    \draw[black,thick,->]  (-3,-9.3) -- (-2,-10.3);
    \draw[black,thick,->]  (-3.9,-9.3) -- (-4.8,-10.3);

    \draw[gray,<->,thick,dashed] (2.2,0.2) arc (220:500:0.7);
    \draw[gray,<->,thick,dashed] (-0.8,-4.5) to[bend right] (1.8,-4.5);
    \draw[gray,<->,thick,dashed,rounded corners] (-2.5,-8.5) -- (-0.8,-9.5) -- (1.8,-9.5) -- (3.5,-8.5);
    \draw[gray,<->,thick,dashed] (2.2,-7.4) arc (270:550:0.7);
    \draw[gray,<->,thick,dashed] (-0.8,-12.5) to[bend right] (1.8,-12.5);
    \draw[gray,<->,thick,dashed,rounded corners] (-4.5,-12.5) -- (-2.8,-14) -- (3.8,-14) -- (5.5,-12.5);
  \end{tikzpicture}
\end{center}
\caption{The partial order for Borel orbits in the space $\x$ of lines transverse to hyperplanes in $\rr^4$.  Complete flags are represented in red, points in $\x$ are represented in blue. Black arrows generate the partial order $\leq$, and dashed gray arrows indicate the involution given by $w_0$.}
\label{fig: lines hyperplanes R4}
\end{figure}

We now go back to the general theory, more examples will be analysed in Examples \ref{ex: fat ideals complementary subspaces and projective} and \ref{ex: fat ideals quadratic forms and projective}.

When $\h=\p_{\theta'}$ is a parabolic subgroup of $\g$, the identification $\p_\theta\backslash\g/\h\cong\w_{\theta,\theta'}$ of Example \ref{ex: partial order when h parabolic} induces an action of the longest element $w_0\in\w$ on the set of relative positions: as $\iota(\theta)=\theta$, the group $\p_\theta\cap\w$ is normalized by $w_0$.
This action is order reversing and has been well studied \cite{KLPdomains,SteIDEALS,SteTreORIENTED}.
In this paper we focus on different classes of homogeneous spaces $\x$, paying special attention to symmetric spaces.
Provided $\h$ is $\tau$-invariant (which is not the case when $\h$ is parabolic but can be assumed to hold when $\h$ is symmetric), the $w_0$-action still makes sense, but is order preserving instead as we now show (see Figure \ref{fig: lines hyperplanes R4} for a picture of how this action looks like in a concrete example).

\begin{prop}\label{prop: action of w0}
Assume that the Cartan involution $\tau$ is chosen in such a way that $\tau(\h)=\h$.
Then the longest element $w_0$ of the Weyl group induces an involution of $\p_\theta\backslash\g/\h$, which is order preserving.
\end{prop}

\begin{proof}
Let $\bfp\in\p_\theta\backslash\g/\h$.
By the Iwasawa Decomposition, we may write $\bfp=\p_\theta k\h$ for some $k\in\ko$.
We let
\[w_0\cdot\bfp:=\p_\theta w_0k\h.\]
\noindent Observe that this is well defined, independently on which representative for $w_0$ we pick in $\n_\ko(\liea)$.
Moreover, if we write $\bfp=\p_\theta k'\h$ with $k'\in\ko$, then by Remark \ref{rem: action of tau and w0 on lieptheta} we have
\[k'=\tau(k')\in\tau(\p_\theta k\h)=w_0\p_\theta w_0k\tau(\h)=w_0\p_\theta w_0k\h,\]
as $\tau$ preserves $\h$ by assumption.
We then have $\p_\theta w_0k'\h=\p_\theta w_0k\h$, showing that the $w_0$-action on $\p_\theta\backslash\g/\h$ is well defined.

The same argument shows that $w_0$ preserves the partial order.
Indeed, let $\bfp=\bor k\h$ and $\bfp'=\bor k'\h$ be two relative positions with $k,k'\in\ko$, and so that $k\in \overline{\bor k'\h}$.
We have
\[k=\tau(k)\in\overline{\tau(\p_\theta k'\h)}=\overline{w_0\p_\theta w_0 k'\h}.\]
Hence
\[w_0k\in\overline{\p_\theta w_0 k'\h},\]
proving $w_0\cdot \bfp\leq w_0\cdot \bfp'$ as claimed.
\end{proof}

\subsection{Transversely related positions}\label{subsec: transversely related}

Our construction of domains of discontinuity is based on the existence of a symmetric relation on the set of relative positions, which relates pairs of relative positions determined by pairs of transverse flags in $\f_\theta$.
Let $\bfp$ and $\bfp'$ be two relative positions.
We write $\bfp\leftrightarrow\bfp'$ and call these positions \textit{transversely related} if there exist transverse flags $\xi$ and $\xi'$ in $\f_\theta$ and a point $x\in\x$ so that $$\pos(\xi,x)=\bfp \tn{ and } \pos(\xi',x)=\bfp'.$$

Alternatively we have the following description, which allows us to prove useful properties.

\begin{lem}\label{lem: relation with triple cosets}
Let $\bfp=\p_\theta g\h$ and $\bfp'=\p_\theta g'\h$ be two relative positions.
Then $\bfp\leftrightarrow\bfp' $ if and only if $g\in\p_\theta w_0 \p_\theta g'\h$.
\end{lem}

\begin{proof}
As $\p_\theta$ and $w_0\p_\theta$ are transverse and the $\g$-action is transitive on the space of pairs of transverse flags in $\f_\theta$, we have $\bfp\leftrightarrow\bfp'$ if and only if there is some $x\in\x$ so that $$\pos(\p_\theta,x)=\bfp \tn{ and } \pos(w_0\p_\theta,x)=\bfp'.$$
\noindent Equivalently, there exist $p,p'\in\p_\theta$ so that $pg\h= x= w_0p'g'\h.$
This finishes the proof.
\end{proof}

As a consequence of Lemma \ref{lem: relation with triple cosets} we have the following property, which states that ``moving up" in the partial order $\leq$ preserves the relation $\leftrightarrow$.

\begin{cor}\label{cor: increasing position keeps the relation}
Let $\bfp$ and $\bfp'$ be two relative positions so that $\bfp\leftrightarrow\bfp'$. Then for every $\bfp''$ satisfying $\bfp''\geq \bfp'$ we have $\bfp\leftrightarrow\bfp''$.
\end{cor}

\begin{proof}
Pick representatives $g,g'$ and $g''$ in $\g$ for $\bfp,\bfp'$ and $\bfp''$ respectively.
As $\bfp\leftrightarrow\bfp'$, by Lemma \ref{lem: relation with triple cosets} we have $\p_\theta g'\h\subset \p_\theta w_0 \p_\theta g\h$.
As $\bfp'\leq \bfp''$, we also have $\p_\theta g'\h\subset\overline{\p_\theta g''\h}$.
In particular, $\p_\theta w_0 \p_\theta g\h$ intersects $\overline{\p_\theta g''\h}$.

On the other hand, by Lemma \ref{lem: maximal positions and openess} (applied to $\h=\p_\theta$), the set $\p_\theta w_0\p_\theta$ is open in $\g$.
Hence, so is $\p_\theta w_0 \p_\theta g\h$.
We then conclude that $(\p_\theta w_0 \p_\theta g\h)\cap (\p_\theta g''\h)\neq\emptyset$.
It follows that $g''\in \p_\theta w_0 \p_\theta g\h$, finishing the proof.
\end{proof}

When $\tau(\h)=\h$, the next two corollaries are useful to understand pairs of transversely related positions in concrete examples (c.f. Examples \ref{ex: fat ideals complementary subspaces and projective} and \ref{ex: fat ideals quadratic forms and projective} below).

\begin{cor}\label{cor: w0 action and transversely related}
Assume that the Cartan involution is chosen in such a way that $\tau(\h)=\h$.
Then for every $\bfp\in\p_\theta\backslash\g/\h$ we have $\bfp\leftrightarrow w_0\cdot\bfp$.
\end{cor}

\begin{proof}
Represent $\bfp=\p_\theta k\h$ for some $k\in\ko$.
As $w_0k \in \p_\theta w_0 \p_\theta k\h$, the result follows from Proposition \ref{prop: action of w0} and Lemma \ref{lem: relation with triple cosets}.
\end{proof}

If moreover there is a unique minimal position, the relation $\leftrightarrow$ is trivial.

\begin{cor}\label{cor: transversely related when unique minimal position}
Suppose that the Cartan involution $\tau$ satisfies $\tau(\h)=\h$.
Assume furthermore that there exists a unique minimal relative position in $\p_\theta\backslash\g/\h$.
Then $\bfp\leftrightarrow\bfp'$ for every pair of relative positions $\bfp$ and $\bfp'$.
\end{cor}

\begin{proof}
Let $\bfp_{\min}$ be the unique minimal position.
By Proposition \ref{prop: action of w0} we have $w_0\cdot\bfp_{\min}=\bfp_{\min}$.
Hence, Corollary \ref{cor: w0 action and transversely related} implies $\bfp_{\min}\leftrightarrow\bfp_{\min}$.
Corollary \ref{cor: increasing position keeps the relation} finishes the proof.
\end{proof}

\begin{rem}\label{rem: transversely related for flags}
When $\h=\p_{\theta'}$ is a parabolic subgroup of $\g$, we do not have $\tau(\h)=\h$.
Nevertheless, as discussed in Subsection \ref{subsec: relative positions and order} the $w_0$-action still makes sense in that setting and actually one has \begin{equation}\label{eq: equivalence transversely related in flag case}
\bfp\leftrightarrow\bfp' \Leftrightarrow \bfp'\geq w_0\cdot \bfp
\end{equation}
\noindent (see \cite[Example 5.1.4]{SteThesis}).
In particular, we still have $\bfp\leftrightarrow w_0\cdot\bfp$ for every relative position $\bfp$.
However, since the unique minimal position is not fixed by $w_0$, Corollary \ref{cor: transversely related when unique minimal position} no longer holds when $\h=\p_{\theta'}$.

On the other hand, in the more general setting we are interested in the equivalence (\ref{eq: equivalence transversely related in flag case}) is not true anymore.
For instance, in the example of Figure \ref{fig: intro} (right) there are two relative positions $\mathbf{p} \neq \mathbf{p}' $ which are not minimal nor maximal.
These two positions are fixed by $w_0$ but thanks to Corollary \ref{cor: transversely related when unique minimal position} we have $\bfp\leftrightarrow\bfp'$.
\end{rem}

\subsection{Minimal transversely related positions for minimal parabolic orbits}\label{subsec: minimal transversely related}

We now carefully analyse pairs of minimal transversely related positions for minimal parabolic orbits on some symmetric spaces (Corollary \ref{cor: related minimal positions} below).
More precisely, in this subsection we assume that $\h$ is symmetric (such that $\tau(\h)=\h$) and take $\theta=\Delta$, so that $\p_\theta = \bor$.
We consider a $\sigma$-invariant Cartan subspace $\liea$ so that $\liea_\h:=\liea\cap\lieh$ is a Cartan subspace of $\h$.
We also assume that $\liea_\h$ contains a regular element and we pick $\liea^+$ so that $\liea_\h\cap\tn{int}(\liea^+)$ is non-empty.
Recall that by Remark \ref{rem: sigma compatible and regular} this implies that $\liea^+$ is $\sigma$-compatible.
Recall also that $\w^\sigma$ denotes the stabilizer of $\liea_\h$ in $\w$.
The following theorem is useful.

\begin{thm}[Matsuki {\cite[§3, Proposition 2]{MatsukiMinimalParabolics}}, see also {\cite[Proposition 7.1.8]{Sch}}]\label{thm: matsuki minimal}
Assume that $\Sigma^+$ is $\sigma$-compatible.
Then a relative position $\bfp \in\bor\backslash\g/\h$ is minimal if and only if there is some $w\in\w^\sigma$ so that $\bfp=\bor w\h$.
In particular, the set of minimal relative positions is parametrized by $\w^\sigma/(\w\cap\h)$.
\end{thm}

Even though the equivalence $\bfp\leftrightarrow\bfp' \Leftrightarrow \bfp'\geq w_0\cdot \bfp$ \noindent does not hold in general (Remark \ref{rem: transversely related for flags}), we can prove the following in our current setting.

\begin{prop}\label{prop: RAMR when one is minimal}
Assume that $\h$ is symmetric and that $\liea_\h$ intersects $\tn{int}(\liea^+)$.
Let $\bfp$ be any element in $\bor\backslash\g/\h$ and suppose that there is some minimal position $\bfp_{\min}$ so that $\bfp\leftrightarrow \bfp_{\min}$.
Then $\bfp\geq  w_0\cdot \bfp_{\min}$.
\end{prop}

\begin{proof}
Write $\bfp=\bor g\h$ and $\bfp_{\min}=\bor w\h$ for some $g\in\g$ and $w\in\w^\sigma$.
We want to show that $w_0w\in\overline{\bor g\h}$ or, equivalently, $w\in\overline{\bor^-w_0 g\h}$.

Now, since $\bfp\leftrightarrow \bfp_{\min}$ we find $x\in\x$ such that $$\pos(\bor,x)=\bor w\h \tn{ and }\pos(\bor^-,x)=\bor g \h.$$
\noindent There exist then $b,b'\in\bor$ such that $bw\cdot o=x=w_0b'g\cdot o$.
That is, we find $b^-\in\bor^-$ such that $$bw\cdot o=b^-w_0g\cdot o.$$
\noindent Hence, $\bor^-w_0g\h=\bor^-bw\h$ and Lemma \ref{lem: for RAMR when one is minimal} below finishes the proof.
\end{proof}

\begin{lem}\label{lem: for RAMR when one is minimal}
For every $w\in\w^\sigma$ and $b\in\bor$ as above one has $w\in\overline{\bor^-bw\h}$.
\end{lem}

\begin{proof}
We may assume $b=n\in\n$.
Since $\liea_\h$ contains a regular element, we may take a sequence $h_p\in\exp(\liea_\h)$ such that $\alpha(\mu(h_p))\to\infty$ for every $\alpha\in\Sigma^+$.
Then $h_pnh_p^{-1}\to 1$.
Furthermore, since $w$ preserves $\liea_\h$ we have $h_pw=wh_p'$ for some $h_p'\in\h$.
Hence, as $h_p\in\exp(\liea_\h)\subset\bor^-$, $$\bor^-bw\h=\bor^-h_pnh_p^{-1}h_pw\h=\bor^-h_pnh_p^{-1}wh_p'\h=\bor^-h_pnh_p^{-1}w\h$$
\noindent and by letting $p\to\infty$ the lemma follows.
\end{proof}

\begin{cor}\label{cor: related minimal positions}
Assume that $\h$ is symmetric and that $\liea_\h$ intersects $\tn{int}(\liea^+)$.
Two minimal positions $\bfp_{\min}$ and $\bfp_{\min}'$ in $\bor\backslash\g/\h$ satisfy $\bfp_{\min}\leftrightarrow \bfp_{\min}'$ if and only if $\bfp_{\min}'=w_0\cdot \bfp_{\min}$.
\end{cor}

\begin{proof}
By Proposition \ref{prop: action of w0} we have $\bfp_{\min}\leftrightarrow w_0\cdot \bfp_{\min}$.
Conversely, suppose that two minimal positions $\bfp_{\min}$ and $\bfp_{\min}'$ are transversely related.
By Proposition \ref{prop: RAMR when one is minimal} we have $\bfp_{\min}'\geq w_0\cdot \bfp_{\min}$.
Since $\bfp_{\min}' $ is minimal the result follows.
\end{proof}

\begin{ex}\label{ex: matsuki and w0 action complementary subspaces}
  Let $\x = \{(U^+,U^-)\in\grass_p(V)\times\grass_{q}(V) \mid U^+\oplus U^-=V  \}$ be as in Example \ref{ex: rel position pairs of complementary subspaces}, with basepoint $o=(U_o^+,U_o^-)$. 
  Recall that $\liea\subset\lieh$, in particular $\w^\sigma = \w$.
  Note however that a Weyl group element $w$ belongs to $\w \cap \h$ if and only if it preserves $U_o^\pm$.
  If we identify the Weyl group $\w$ with permutations of $\{1, \dots, d\}$, then $w \in \w \cap \h$ if and only if $w$ preserves $\{1,\dots,p\}$.
  Theorem \ref{thm: matsuki minimal} allows us to identify the set of minimal positions in $\bor \backslash \g / \h$ with the quotient $\w^\sigma \!/ (\w \cap \h)$.
  It can be represented by subsets $A \subset \{1, \dots, d\}$ with $\#A = p$ (the images of $\{1, \dots, p\}$ under $w \in \w^\sigma$).
  Hence $\#(\w^\sigma / (\w \cap \h)) = \binom{d}{p}$.

  The order preserving involution on $\bor \backslash \g / \h$ restricts to the set of minimal positions and corresponds to the action of $w_0$ on $\w^\sigma \!/ (\w \cap \h)$ by left--multiplication.
  It has no fixed points if $d$ is even and $p$ is odd, and $\binom{\lfloor d/2 \rfloor}{\lfloor p/2 \rfloor}$ fixed points otherwise.
\end{ex}

\section{\texorpdfstring{Fat and $w_0$-fat ideals}{Fat and w0-fat ideals}}\label{sec: fat ideals}

We now generalize the notion of fat ideal introduced by Kapovich-Leeb-Porti \cite{KLPdomains} to our current setting.
We do this in two ways.
On the one hand, for a general closed subgroup $\h$ we use the relation $\leftrightarrow$.
On the other, when $\h$ is $\tau$-invariant we use the involution $w_0$ of Proposition \ref{prop: action of w0} to define a more general notion, which will be important to describe maximal domains of discontinuity in Section \ref{sec: maximality}.

\subsection{Definition and first properties}

Throughout this section we fix a self-opposite parabolic subgroup $\p_\theta$, for some non-empty subset $\theta\subset\Delta$ and a closed subgroup $\h$ of $\g$ for which the double coset space $\p_\theta\backslash\g/\h$ is finite.
Recall that an \textit{ideal} is a subset $\bfi\subset\p_\theta\backslash\g/\h$ so that for every $\bfp\in \bfi$ and every $\bfp'\leq \bfp$, one has $\bfp'\in \bfi$.

\begin{dfn}\label{dfn: fat and w0fat ideal}
An ideal $\bfi$ is \textit{fat} if for every $\bfp\notin \bfi$ there exists $\bfp'\in \bfi$ so that $\bfp\leftrightarrow\bfp'$.
In the case $\tau(\h)=\h$, we will say that $\bfi$ is $w_0$-\textit{fat} if for every minimal position $\bfp_{\min}\notin \bfi$ we have $w_o\cdot\bfp_{\min}\in\bfi$.
\end{dfn}

Let us emphasize that in the case $\tau(\h)=\h$ we do not require the existence of a minimal position not belonging to $\bfi$.
More precisely, an ideal containing all minimal positions is $w_0$-fat. 
Observe also that if an ideal $\bfi'$ contains a fat (resp. $w_0$-fat) ideal $\bfi$, then $\bfi'$ is itself fat (resp. $w_0$-fat).
This is why we will be typically interested in finding the minimal (non-empty) fat and $w_0$-fat ideals.

\begin{ex}\label{ex: fat ideals in flag case}
Definition \ref{dfn: fat and w0fat ideal} takes inspiration from Kapovich-Leeb-Porti \cite{KLPdomains}.
In that work the authors define a fat ideal of $\p_\theta\backslash\g/\p_{\theta'}$ to be an ideal $\bfi$ for which $w_0\cdot\bfp\in \bfi$ for every $\bfp\notin\bfi$.
Thanks to Equation (\ref{eq: equivalence transversely related in flag case}), this is equivalent to our definition of fat ideal.
The relation between fat and $w_0$-fat ideals in Definition \ref{dfn: fat and w0fat ideal} will be discussed in Lemma \ref{lem: fat is w0-fat} and Example \ref{ex: fat ideals lines transverse to hyperplanes and full flags}.
\end{ex}

Provided there is more than one relative position, fat and $w_0$-fat ideals exist.

\begin{prop}\label{prop: non maximal is fat}
Assume that $\vert\p_\theta\backslash\g/\h\vert >1$.
Then $$\Inonmax:=\{\bfp\in\p_\theta\backslash\g/\h: \bfp \tn{ is not maximal}\}$$ \noindent is a fat ideal (and $w_0$-fat when $\tau(\h)=\h$).
\end{prop}

\begin{proof}
Since $\g$ is connected and $\p_\theta\backslash\g/\h$ contains at least two elements, non-maximal positions do exist (c.f. Lemma \ref{lem: maximal positions and openess}).
It is also clear that $\Inonmax$ is an ideal.

Suppose by contradiction that there exists a maximal relative position $\bfp$ which is not transversely related to any non-maximal position.
Fix a flag $\xi\in\f_\theta$ so that $\pos(\xi,o)$ is not maximal.
Then for every flag $\xi'$ transverse to $\xi$ we have $\pos(\xi',o)\neq \bfp$.
Since the set of flags in $\f_\theta$ which are transverse to $\xi$ is dense in $\f_\theta$, we obtain a contradiction by applying Lemma \ref{lem: maximal positions and openess}.
\end{proof}

The following is a consequence of Corollary \ref{cor: transversely related when unique minimal position}.

\begin{cor}\label{cor: minimal fat ideal when unique minimal position}
Suppose that the Cartan involution $\tau$ satisfies $\tau(\h)=\h$.
Assume furthermore that there exists a unique minimal relative position $\bfp_{\min}\in\p_\theta\backslash\g/\h$.
Then every non--empty ideal in $\p_\theta\backslash\g/\h$ is fat and $w_0$-fat.
\end{cor}

For minimal parabolic orbits on some symmetric spaces we can show that fat ideals are always $w_0$-fat (the converse does not hold, see Example \ref{ex: fat ideals lines transverse to hyperplanes and full flags}).

\begin{lem}\label{lem: fat is w0-fat}
Suppose that $\h$ is symmetric, $\liea$ is $\sigma$-invariant, and $\liea_\h:=\liea\cap\lieh$ is a Cartan subspace of $\h$.
Assume that $\liea_\h$ contains a regular element, and pick the Weyl chamber $\liea^+$ in such a way that $\tn{int}(\liea^+)\cap\liea_\h$ is non-empty.
Then every fat ideal $\bfi\subset\bor\backslash\g/\h$ is $w_0$-fat.
\end{lem}

\begin{proof}
Let $\bfp_{\min}$ be a minimal relative position so that $\bfp_{\min}\notin\bfi$.
Since $\bfi$ is fat, there is some $\bfp\in\bfi$ so that $\bfp\leftrightarrow\bfp_{\min}$.
Proposition \ref{prop: RAMR when one is minimal} implies $\bfp\geq w_0\cdot\bfp_{\min}$.
Since $\bfi$ is an ideal, this shows $w_0\cdot\bfp_{\min}\in\bfi$.
\end{proof}

\subsection{Examples}\label{subsec: ex of fat ideals}
We now discuss some examples of (non-empty minimal) fat and $w_0$-fat ideals.

\begin{ex}\label{ex: fat ideals lines transverse to hyperplanes and full flags}
Let $\x$ and $\theta=\Delta$ be as in Example \ref{ex: rel position pairs of complementary subspaces} for $p=1$ and $q=3$, that is, $\x$ is the space of pairs consisting on a line transverse to a hyperplane in $\rr^4$.
Then $\bfi=\{\bor\h,\bor w_0w_{13}\h\}$ is a $w_0$-fat ideal, but it is not fat (the position right above $\bor\h$ and $\bor w_0w_{13}\h$ is not transversely related to either $\bor\h$ or $\bor w_0 w_{13} \h$).
All fat and $w_0$-fat ideals are easily computed out of Figure \ref{fig: lines hyperplanes R4}.
\end{ex}

\begin{ex}\label{ex: fat ideals group manifolds}
Let $\x=\g_0$ be a group manifold, as in Example \ref{ex: relative positions in group manifolds}.
By Corollary \ref{cor: minimal fat ideal when unique minimal position} the ideal $\Imin:=\{\bfp_{\min}\}$ is $w_0$-fat (for every possible choice of $\theta$), and so is every ideal.
\end{ex}

\begin{ex}\label{ex: fat ideals Hpq}
Consider the case of Example \ref{ex: relative positions in hpq}, that is, we let $\x=\mathbb{H}^{p,q-1}$ and $\p_\theta=\p_1^{p,q}$, the stabilizer of an isotropic line in $\rr^{p+q}$.
There is only one non--trivial ideal, which is $w_0$-fat by Corollary \ref{cor: minimal fat ideal when unique minimal position}.
\end{ex}

\begin{ex}\label{ex: fat ideals complementary subspaces and projective}
Let $\x=\{(U^+,U^-)\in\mathscr{G}_p(V)\times\mathscr{G}_q(V): U^+\oplus U^-=V\}$ be as in Example \ref{ex: rel position pairs of complementary subspaces}, for general $1\leq p \le q$.
Let us compute the minimal non-empty $w_0$-fat ideals for the case $\theta=\{\alpha_1,\alpha_{d-1}\}$, that is, where $\f_\theta$ is the space of pairs $(\xi^1,\xi^{d-1})$ in $\mathbb{P}(V)\times\mathbb{P}(V^*)$ so that $\xi^1\subset\xi^{d-1}$.

The relative position between $(U^+, U^-)$ and $(\xi^1, \xi^{d-1})$ is minimal if and only if $\xi^1$ is contained in either $U^+$ or $U^-$ and $\xi^{d-1}$ contains either $U^+$ or $U^-$.
If we assume $p > 1$ this leaves us with four different minimal positions, which we can describe schematically as
\[\bfp_1 \coloneqq \{\xi^1\subset U^+\subset\xi^{d-1}\}, \quad \bfp_2 \coloneqq \{(\xi^1\subset U^+) \tn{ }\wedge\tn{ } (U^-\subset\xi^{d-1})\},\]
\[\bfp_3 \coloneqq \{(\xi^1\subset U^-)\tn{ }\wedge \tn{ }(U^+\subset\xi^{d-1})\}, \quad \bfp_4 \coloneqq \{\xi^1\subset U^-\subset\xi^{d-1}\}.\]
Positions $\bfp_2$ and $\bfp_3$ are fixed by $w_0$, while $\bfp_1$ and $\bfp_4$ are permuted.
We can see this from Theorem \ref{thm: matsuki minimal} or alternatively by identifying which positions are transversely related ($\bfp_1 \leftrightarrow \bfp_4$, $\bfp_2 \leftrightarrow \bfp_2$ and $\bfp_3 \leftrightarrow \bfp_3$ and no others) and then applying Corollary \ref{cor: w0 action and transversely related}.
As a consequence, the minimal $w_0$-fat ideals are
\[\bfi_{123} \coloneqq \{\bfp_1, \bfp_2, \bfp_3\} \text{ and } \bfi_{234} \coloneqq \{\bfp_2, \bfp_3, \bfp_4\}.\]
Note that any ideal containing one of these is also $w_0$-fat, in particular the ideal $\Imin = \{\bfp_1, \bfp_2, \bfp_3, \bfp_4\}$ containing all minimal relative positions.

On the other hand, if $p = 1$ and $q > 1$ then the position $\bfp_2$ does not exist and $\bfp_1$ and $\bfp_4$ can be more easily described as
\[\bfp_1 = \{\xi^1 = U^+\}, \qquad \bfp_4 = \{U^- = \xi^{d-1}\}.\]
So we have three positions, one of which, $\bfp_3$, is fixed by the $w_0$ action, and the two minimal $w_0$-fat ideals are $\bfi_{13} \coloneqq \{\bfp_1, \bfp_3\}$ and $\bfi_{34} \coloneqq \{\bfp_3, \bfp_4\}$.
\end{ex}

\begin{ex}\label{ex: fat ideals quadratic forms and projective}
  Choose a symmetric bilinear form of signature $(p,q)$ on a $(p+q)$--dimensional real vector space $V$.
  Let $\g = \PSL(V)$ and $\h = \mathsf{PSO}(p,q)$ the subgroup preserving the form.
  Then $\x = \g/\h$ is a symmetric space and can be identified with the space of signature $(p,q)$ forms on $V$, up to scaling.
  We consider again the case $\theta=\{\alpha_1,\alpha_{d-1}\}$.
  A picture of the Bruhat order in the lowest dimensional interesting example, the case $(p,q) = (1,2)$, is shown in Figure \ref{fig: intro} on the right.
  
  For arbitrary $p$ and $q$, there is a unique minimal position in $\p_\theta\backslash\g / \h$, represented by $(\xi^1,\xi^{d-1})\in\f_\theta$ and $x \in \x$ such that $\xi^1$ is isotropic for the quadratic form $x$, and $\xi^{d-1}$ is the orthogonal complement of $\xi^1$ with respect to this form.
  By Corollary \ref{cor: dod if unique minimal position} every non--empty ideal is $w_0$-fat.
\end{ex}



\section{Discrete subgroups and Anosov representations}\label{sec: anosov}

  In this section we recall the notion of Anosov representations and summarize their main properties.
  All the results presented here are well known and due to Benoist \cite{BenPAGL,BenPAGLII}, Labourie \cite{Lab} and Guichard-Wienhard \cite{GW}, perhaps with the exception of Proposition \ref{prop: directions interior to limit cone approached by mu(Gamma)} (which is a central ingredient in the proof of Theorem \ref{thm: maximality in intro}).
  Subsections \ref{subsec: tits} and \ref{subsec: theta gap and proximality} are intended to fix some terminology needed to understand dynamical and asymptotic properties of elements of $\g$ acting on $\f_\theta$.
  Anosov representations are introduced in Subsection \ref{subsec: anosov reps}.
  In Subsection \ref{subsec: limit cone} we recall the definition and central properties of Benoist's \textit{limit cone}, and also prove Proposition \ref{prop: directions interior to limit cone approached by mu(Gamma)}. 
  In Subsection \ref{subsec: examples of anosov reps} we discuss important examples.

\subsection{Representations of $\g$}\label{subsec: tits}

Tits representations are crucial to understand both the dynamics of discrete subgroups of $\g$, as well as their quantitative properties.

Let $V$ be a finite dimensional real vector space and $\Lambda:\g\to\mathsf{GL}(V)$ be an irreducible representation with derivative $\mathrm d_1\Lambda \colon \lieg \to \mathfrak{gl}(V)$.
A \textit{weight} of $\Lambda$ is a functional $\chi\in\liea^*$ so that the \textit{weight space}
\[V_\chi:=\{v\in V \colon \mathrm{d}_1\Lambda(A)(v) = \chi(A)v \tn{ for all } A\in\liea\}\]
is non-zero.
The set of weights of $\Lambda$ carries a partial order defined by $$\chi\geq \chi' \Leftrightarrow \chi-\chi'=\displaystyle\sum_{\alpha\in\Delta}a_\alpha\alpha,$$
\noindent with $a_\alpha\geq 0$ for all $\alpha$.
A theorem of Tits \cite{Tits} states that there exists a unique weight $\chi_\Lambda$ which is maximal with respect to this order, called the \textit{highest weight} of $\Lambda$.
The representation $\Lambda$ is said to be \textit{proximal} if $V_{\chi_\Lambda}$ is one dimensional.
One also has the following useful proposition.

\begin{prop}[Tits {\cite{Tits}}]\label{prop: tits}
For every $\alpha\in\Delta$ there exists a finite dimensional real vector space $V_\alpha$ and an irreducible proximal representation $\Lambda_\alpha:\g\to\mathsf{GL}(V_\alpha)$, so that the mapping $$\liea\to\rr^{\#\Delta} \colon A\mapsto(\chi_{\Lambda_\alpha}(A))_{\alpha\in\Delta}$$\noindent is an isomorphism.
\end{prop}

\begin{ex}
Let $\g$ be as in Example \ref{ex: flags sld}.
Simple roots correspond to the choice of an integer $1\leq k\leq d-1$, and a set of representations as in Proposition \ref{prop: tits} is in this case the set of $k^{\tn{th}}$-exterior powers of $V$.
\end{ex}

We fix from now on a set of representations $\{\Lambda_\alpha\}_{\alpha\in\Delta}$ as in Proposition \ref{prop: tits} and let $\theta\subset\Delta$ be non-empty.
The representations $\Lambda_\alpha$ induce $\g$-equivariant maps $\f_\theta\to\mathbb{P}(V_\alpha)$ which we also denote by $\Lambda_\alpha$.
Taking duals, we also have $\g$-equivariant maps $\Lambda_\alpha^- \colon \f_\theta^-\to\mathbb{P}(V_\alpha^*)$.
Two flags in $\f_\theta$ (resp. $\f_\theta^-$) coincide if and only if their images by $\Lambda_\alpha$ (resp. $\Lambda_\alpha^-$) coincide for every $\alpha\in\theta$.
A flag $\xi_+\in\f_\theta$ is transverse to $\xi_-\in\f_\theta^-$ if and only if the line $\Lambda_\alpha(\xi_+)$ is not contained in the hyperplane $\Lambda_\alpha^-(\xi_-)$, for every $\alpha\in\theta$.

For each $\alpha\in\theta$, we fix a $\ko$-invariant Euclidean norm $\Vert\cdot\Vert_\alpha$ (resp. $\Vert\cdot\Vert_\alpha^*$) on $V_\alpha$ (resp. $V_\alpha^*$) so that $\Lambda_\alpha(\exp(A))$ (resp. $\Lambda_\alpha^-(\exp(A))$) is self-adjoint for every $A\in\liea$ (c.f. \cite[Lemme 2.2]{BenPAGLII}).
This defines a $\ko$-invariant distance $d_\alpha(\cdot,\cdot)$ on $\mathbb{P}(V_\alpha)$ (resp. $d_\alpha^-(\cdot,\cdot)$ on $\mathbb{P}(V_\alpha^*)$).
Given $\varepsilon>0$, we let $$b_\varepsilon(\xi_+):=\{\xi \in\f_\theta: d_\alpha(\Lambda_\alpha(\xi),\Lambda_\alpha(\xi_+))\leq \varepsilon  \tn{ for all } \alpha\in\theta\}$$ \noindent and $$B_\varepsilon(\xi_-):=\{\xi\in\f_\theta: d_\alpha(\Lambda_\alpha(\xi),\Lambda_\alpha^-(\xi_-))\geq \varepsilon \tn{ for all } \alpha\in\theta\} .$$\noindent In the last equality, $\Lambda_\alpha^-(\xi_-)$ is seen as a compact subset of $\mathbb{P}(V_\alpha)$ by identifying $\mathbb{P}(V_\alpha^*)$ with the space of linear hyperplanes of $V_\alpha$.

\subsection{$\theta$-gap and proximality}\label{subsec: theta gap and proximality}
We now turn to the dynamics of elements $g\in\g$ acting on $\f_\theta$.
We say that $g$ has a \textit{gap of index $\theta$} (or a $\theta$-\textit{gap}) if $\alpha(\mu(g))>0$ for all $\alpha\in\theta$.
In that case, if $g=k\exp(\mu(g))l$ is a Cartan decomposition of $g$ we let $$U_\theta(g):=k\p_\theta\in\f_\theta \tn{ and } S_\theta(g):=l^{-1}\p_\theta^-\in\f_\theta^-.$$
\noindent We call $U_\theta(g)$ (resp. $S_\theta(g)$) the \textit{Cartan attractor} (resp. \textit{Cartan repellor}) of $g$.
Note that these two flags are not necessarily transverse, nor fixed by $g$.
When $\theta=\Delta$ we simply denote $U(g):=U_\Delta(g)$ and $S(g):=S_\Delta(g)$.
The following remark is classical (see e.g. \cite[Lemma 3.1.3]{SteThesis} for a proof).

\begin{rem}\label{rem: dynamics g cartan attractor and repellor}
Fix a positive $\varepsilon$. There exists $L>0$ so that for every $g\in\g$ satisfying $\displaystyle\min_{\alpha\in\theta}\alpha(\mu(g))>L$ one has $$g\cdot B_\varepsilon(S_\theta(g))\subset b_\varepsilon(U_\theta(g)).$$

\end{rem}

The \textit{Jordan projection} of $\g$ is the map $\lambda:\g\to\liea^+$ defined by $$\lambda(g):=\displaystyle\lim_{p\to\infty}\frac{\mu(g^p)}{p}.$$
\noindent The element $g\in\g$ is said to be \textit{$\theta$-proximal} (or \textit{proximal on $\f_\theta$}) if $\alpha(\lambda(g))>0$ for all $\alpha\in\theta$.
In this case, the action of $g$ on $\f_\theta$ (resp. $\f_\theta^-$) has a fixed point $g_+=g_+^\theta$ (resp. $g_-=g_-^\theta$) so that $g_+$ is transverse to $g_-$ and $$\displaystyle\lim_{p\to\infty}g^p\cdot \xi= g_+$$ \noindent for every $\xi\in\f_\theta$ transverse to $g_-$. 
Explicitly, for every $\alpha\in\theta$, $\Lambda_\alpha(g_+)$ is the eigenline of $\Lambda_\alpha(g)$ associated to the highest eigenvalue of $\Lambda_\alpha(g)$, and $\Lambda_\alpha^-(g_-)$ is the complementary invariant hyperplane.

One also has the following quantified version of proximality which is useful for estimates.
Given $0<\varepsilon\leq r$, a $\theta$-proximal element $g\in\g$ is said to be $(r,\varepsilon)$-\textit{proximal} if $$d_\alpha(\Lambda_\alpha(g_+),\Lambda_\alpha^-(g_-))\geq 2r$$ \noindent for every $\alpha\in\theta$, and $g\cdot B_\varepsilon(g_-)\subset b_\varepsilon(g_+)$.

In the rest of the subsection we focus on estimating the Cartan projection of products of proximal elements (c.f. Lemma \ref{lem: product of repsilon loxodromic} below).
Even though the discussion that follows can be done for general $\theta\subset\Delta$, we will only apply it in the case $\theta=\Delta$.
Hence we restrict our attention to that case, which makes things easier to state.

A $\Delta$-proximal element $g$ is sometimes called \textit{loxodromic}.
It is called $(r,\varepsilon)$-\textit{loxodromic} if $\Lambda_\alpha(g)$ is $(r,\varepsilon)$-proximal for every $\alpha\in\Delta$.

An important object when describing the Cartan projection of a product is a vector $[g_0,\dots,g_s]\in\liea$ which quantifies how far the repelling and attracting fixed points of $g_0,\dots,g_s\in\g$ are from each other, for a given sequence $g_0,g_1,\dots,g_s\in\g$ of loxodromic elements with $(g_0)_\pm,\dots,(g_s)_\pm$ in general position.
As we don't really need the definition of $[g_0,\dots,g_s]$ we refer the interested reader to Benoist \cite[p.3 \& Section 3]{BenPAGLII} for details. 
We record however the following important property of $[g_0,\dots,g_s]$ which, together with Lemma \ref{lem: product of repsilon loxodromic}, is the only thing that we will need about this vector.

\begin{rem}\label{rem: cross ratio depends only on flags}
The vector $[g_0,\dots,g_s]\in\liea$ only depends on the flags $(g_0)_\pm, \dots, (g_s)_\pm$, not on the elements $g_0,\dots,g_s\in\g$ themselves. 
\end{rem}

\begin{lem}[See e.g. Benoist {\cite[Lemme 3.4]{BenPAGLII}}]\label{lem: product of repsilon loxodromic}
Fix $r>0$ and $\delta> 0$.
Then for every small enough $\varepsilon$, the following is satisfied: consider a tuple $g_1,\dots,g_s=g_0$ of $(r,\varepsilon)$-loxodromic elements such that for every $j=0,\dots,s-1$ and every $\alpha\in\Delta$ one has $$d_\alpha(\Lambda_\alpha ((g_{j+1})_+),\Lambda_\alpha^-((g_j)_-))\geq 6r.$$
\noindent Then the product $g=g_1\dots g_s$ is $(2r,2\varepsilon)$-loxodromic with $$d_\alpha(\Lambda_\alpha(g_-),\Lambda_\alpha((g_s)_-))\leq \delta \tn{ and } d_\alpha(\Lambda_\alpha(g_+),\Lambda_\alpha((g_1)_+))\leq \delta$$
\noindent for every $\alpha\in\Delta$.
Furthermore,
$$\left\Vert \mu( g) - (\lambda(g_1)+\dots+\lambda(g_s))+[g_1,\dots,g_s] \right\Vert<\delta.$$
\end{lem}

\subsection{Anosov representations}\label{subsec: anosov reps}

Anosov representations form a stable class of (almost) faithful and discrete representations of word hyperbolic groups into $\g$ providing a framework unifying examples of different nature.
They are nowadays understood as a higher rank generalization of convex co-compact representations into rank one Lie groups.
They were introduced by Labourie \cite{Lab} for fundamental groups of closed negatively curved manifolds, and then extended by Guichard-Wienhard \cite{GW} to general word hyperbolic groups.
The definition that we present here is not the original one, but a simpler equivalent one given in \cite{KLPDynGeomCharacterizations,GGKW,BPS}.

Let $\Gamma$ be a finitely generated group and $\vert\cdot\vert$ be the word length associated to a finite symmetric generating set, which will be fixed from now on.
Let $\theta\subset\Delta$ be a non-empty set.
A representation $\rho:\Gamma\to\g$ is said to be $\theta$-\textit{Anosov} if there exist positive constants $c$ and $C$ so that \begin{equation}\label{eq def anosov}
\alpha(\mu(\rho(\gamma)))\geq c\vert\gamma\vert-C
\end{equation}
\noindent for every $\gamma\in\Gamma$ and $\alpha\in\theta$.

Note that $\theta$-Anosov representations are quasi-isometrically embedded. In particular, they are discrete and have finite kernels.
Further, by Kapovich-Leeb-Porti \cite[Theorem 1.4]{KLPMorse} (see also \cite[Section 3]{BPS}), if $\rho$ is $\theta$-Anosov then $\Gamma$ is word hyperbolic (throughout, we assume that $\Gamma$ is non-elementary).
The Gromov boundary of $\Gamma$ will be denoted by $\bg$, and we also let $\bgc$ be the set of ordered pairs of distinct points in $\bg$.
If $\gamma\in\Gamma$ has infinite order, it has two fixed points on $\bg$, a repelling one (denoted by $\gamma_-$), and an attracting one (denoted by $\gamma_+$).
We let $\gh\subset\Gamma$ be the subset consisting of infinite order elements.
By Gromov \cite[Corollary 8.2.G]{Gro}, the set $\{(\gamma_-,\gamma_+)\}_{\gamma\in\gh}$ is dense in $\bgc$.

Let $\rho\colon\Gamma\to\g$ be $\theta$-Anosov.
One can check that $\rho$ is also $\iota(\theta)$-Anosov.
Hence, we will assume from now on that $\theta=\iota(\theta)$.
Central in the theory is the following property (see \cite{BPS,GGKW,KLPDynGeomCharacterizations}): every $\theta$-Anosov representation $\rho$ admits a \textit{limit map}.
By definition, this is a continuous, $\rho$-equivariant, dynamics-preserving map
\[\xi_{\rho,\theta} \colon \bg\to\f_\theta\]
which is moreover transverse, i.e. $\xi_{\rho,\theta}(z)$ is transverse to $\xi_{\rho,\theta}(z')$ whenever $z \neq z'$.
We recall here that $\xi_{\rho,\theta}$ is said to be \textit{dynamics preserving} if for every $\gamma\in\gh$, the element $\xi_{\rho,\theta}(\gamma_+)$ (resp. $\xi_{\rho,\theta}(\gamma_-)$) is an attractive (resp. repelling) fixed point of $\rho(\gamma)$ acting on $\f_\theta$.
In particular, $\rho(\gamma)$ is proximal on $\f_\theta$ and $$\rho(\gamma)_\pm=\xi_{\rho,\theta}(\gamma_\pm).$$
\noindent As a consequence, the limit maps are injective and uniquely determined by $\rho$.
When $\rho$ is $\Delta$-Anosov, we denote $\xi_\rho:=\xi_{\rho,\Delta}$.

By Guichard-Wienhard \cite[Theorem 5.13]{GW}, the limit map varies continuously with the representation.
The image $\Lambda_\rho^\theta:=\xi_{\rho,\theta}(\bg)$ is sometimes called the $\theta$-\textit{limit set} of $\rho$. It has the following important characterization.

\begin{prop}[See {\cite[Theorem 5.3]{GGKW}} or {\cite[Subsection 3.4]{BPS}}]\label{prop: limit set for anosov with cartan attractors}
Let $\rho:\Gamma\to\g$ be a $\theta$-Anosov representation.
Then $\xi_{\rho,\theta}(\bg)$ coincides with the set of accumulation points of sequences of the form $\{U_\theta(\rho(\gamma_p))\}$, where $\gamma_p\to\infty$.
Furthermore, given a positive $\delta$ one has $$d_\alpha(\Lambda_\alpha(U_\theta(\rho(\gamma))),\Lambda_\alpha(\rho(\gamma)_+))<\delta$$
\noindent for every $\gamma\in\Gamma$ with sufficiently large $\vert\gamma\vert$, and every $\alpha\in\theta$.
\end{prop}

We also have the following lemma.

\begin{lem}[c.f. Sambarino {\cite[Lemma 5.7]{SamQuantitative}}]\label{lem: anosov and repsilon proximality}
Let $\rho:\Gamma\to\g$ be a $\theta$-Anosov representation and fix real numbers $0<\varepsilon\leq r$.
Then there exists a positive $L$ with the following property: for every $\gamma\in\gh$ satisfying $\vert\gamma\vert>L$ and such that $$d_\alpha(\Lambda_\alpha(\rho(\gamma)_+),\Lambda_\alpha^-(\rho(\gamma)_-))\geq 2r$$ \noindent holds for every $\alpha\in\theta$, one has that $\rho(\gamma)$ is $(r,\varepsilon)$-proximal on $\f_\theta$.
\end{lem}

\subsection{Limit cone}\label{subsec: limit cone}

Let $\Xi<\g$ be an infinite discrete subgroup.
The \textit{asymptotic cone} of $\Xi$, denoted by $\conexi$, is the set of limit points of the form $$\displaystyle\lim_{p\to \infty}\frac{\mu(\gamma_p)}{t_p},$$
\noindent for sequences $\{\gamma_p\}\subset\Xi$ and $t_p\to\infty$.
On the other hand, the \textit{limit cone} of $\Xi$ is the smallest closed cone containing the set $\lambda(\Xi)$.
These objects were introduced in foundational work by Benoist \cite{BenPAGL}.
Benoist showed that, when $\Xi$ is Zariski dense, the limit cone and the asymptotic cone coincide. 
Further, $\conexi$ is convex and has non-empty interior.
Lemma \ref{lem: for prop directions interior to limit cone approached by mu(Gamma)} below, which will be used in the proof of Proposition \ref{prop: directions interior to limit cone approached by mu(Gamma)}, also follows from combining several results by Benoist.

\begin{lem}[Benoist \cite{BenPAGL,BenPAGLII}]\label{lem: for prop directions interior to limit cone approached by mu(Gamma)}
Let $\Xi<\g$ be a Zariski dense discrete subgroup.
Fix $X_0\in\tn{int}(\conexi)$ and $\delta>0$.
Then there exist $r>0$, $A\in\liea$ and loxodromic elements $g_0,\dots,g_t\in\Xi$ with the following properties:
\begin{enumerate}
\item for every $i,j=0,\dots,t$ and every $\alpha\in\Delta$ we have $$d_\alpha(\Lambda_\alpha^-((g_i)_-),\Lambda_\alpha((g_j)_+))\geq r,$$
\item the vector $X_0$ belongs to $\tn{int}(\calc)$, where $\calc:=(\rr_{\geq 0})\lambda(g_0)+\dots+(\rr_{\geq 0})\lambda(g_t)$, and
\item the set $\mathbb{N}\lambda(g_0)+\dots+\mathbb{N}\lambda(g_t)$
\noindent is $\delta$-dense in $A+\calc$.
\end{enumerate}
\end{lem}

We have the following consequence for Anosov representations, which refines Proposition \ref{prop: limit set for anosov with cartan attractors} as it allows us to control not only where Cartan attractors and repellors are located, but also the corresponding Cartan projection. We let $d(\cdot,\cdot)$ be the distance on $\liea$ induced by the Killing form of $\lieg$ and $\cone:=\mathcal{L}_{\rho(\Gamma)}$.

\begin{prop}\label{prop: directions interior to limit cone approached by mu(Gamma)}
Let $\rho:\Gamma\to\g$ be a Zariski dense $\Delta$-Anosov representation.
Fix vectors $Y\in\liea$ and $X_0\in\tn{int}(\cone)$, and a pair of limit points $\xi_\pm\in\xi_\rho(\bg)$.
Then for every $\delta>0$ there is an element $\gamma\in\Gamma$ such that $$d_\alpha(\Lambda_\alpha(S(\rho(\gamma))),\Lambda_\alpha(\xi_-))<\delta \tn{ and } d_\alpha(\Lambda_\alpha(U(\rho(\gamma))),\Lambda_\alpha(\xi_+))<\delta$$\noindent  for every $\alpha\in\Delta$, and $d(\mu(\rho(\gamma)),Y+(\rr_{\geq 0}) X_0)<\delta$.
\end{prop}

\begin{proof}

Let $g_0=\rho(\gamma_0),\dots,g_t=\rho(\gamma_t)\in\rho(\Gamma)$, $A\in\liea$ and $\calc$ be as in Lemma \ref{lem: for prop directions interior to limit cone approached by mu(Gamma)}, so in particular $X_0\in\tn{int}(\calc)$.

After perturbing slightly $\xi_-$ if necessary we may assume that $\xi_-$ and $\xi_+$ are transverse.
Since $\{(\gamma_-,\gamma_+)\}_{\gamma\in\gh}$ is dense in $\bgc$, we may take $\widehat{\gamma}\in\gh$ so that \begin{equation} \label{eq: in prop directions interior to limit cone approached by mu(Gamma)}
d_\alpha(\Lambda_\alpha(\rho(\widehat{\gamma})_\pm),\Lambda_\alpha(\xi_\pm))<\delta
\end{equation} \noindent for all $\alpha\in\Delta$.
We may also assume that $$\widehat{\gamma}_-,\widehat{\gamma}_+,(\gamma_0)_-,(\gamma_0)_+,\dots,(\gamma_t)_-,(\gamma_t)_+$$ \noindent are pairwise distinct.
Hence, there exists some $r>0$ so that for every $\alpha\in\Delta$ we have $$ d_\alpha(\Lambda_\alpha^-(\rho(\gamma)_-),\Lambda_\alpha(\rho(\gamma')_+))\geq 6r $$ \noindent for every $\gamma,\gamma'\in \{\widehat{\gamma},\widehat{\gamma}^{-1},\gamma_0,\gamma_0^{-1},\dots,\gamma_t,\gamma_t^{-1}\}$ with $\gamma'\neq\gamma^{-1}$.

Let $\varepsilon$ be as in Lemma \ref{lem: product of repsilon loxodromic} (for this $r$ and our fixed $\delta$).
By Lemma \ref{lem: anosov and repsilon proximality}, there exists some $N>0$ so that for all $n,n_0,\dots,n_t\geq N$ the elements $\rho(\widehat{\gamma}^n),\rho(\gamma_0^{n_0}),\dots,\rho(\gamma_t^{n_t})$ are $(r,\varepsilon)$-loxodromic.
If we write $\gamma=\gamma_{n,n_0,\dots,n_t}:=\widehat{\gamma}^n\gamma_0^{n_0}\dots\gamma_t^{n_t}\widehat{\gamma}^n$,
Lemma \ref{lem: product of repsilon loxodromic} and Equation (\ref{eq: in prop directions interior to limit cone approached by mu(Gamma)}) imply $$d_\alpha(\Lambda_\alpha(\rho(\gamma)_\pm),\Lambda_\alpha(\xi_\pm))<2\delta$$ \noindent for all $\alpha\in\Delta$ and all $n,n_0,\dots,n_t\geq N$.
\noindent Moreover, by Proposition \ref{prop: limit set for anosov with cartan attractors} we may assume that $N$ is large enough, so that $$d_\alpha(\Lambda_\alpha(U(\rho(\gamma))),\Lambda_\alpha(\xi_+))\leq 3\delta \tn{ and } d_\alpha(\Lambda_\alpha(S(\rho(\gamma)),\Lambda_\alpha(\xi_-))\leq 3\delta$$ \noindent holds for every $n,n_0,\dots,n_t\geq N$ and every $\alpha\in\Delta$.
To finish the proof we only have to show that we may pick the exponents $n, n_0,\dots,n_t\geq N$ in such a way that the condition $d(\mu(\rho(\gamma)),Y+(\rr_{\geq 0})X_0)<\delta$ is also satisfied.

First of all, we fix $n:=N$.
Set $$A':=2\lambda(\rho(\widehat{\gamma}^N))-[\rho(\widehat{\gamma}^N),\rho(\gamma_0)^{n_0},\dots,\rho(\gamma_t)^{n_t},\rho(\widehat{\gamma})],$$ \noindent which by Remark \ref{rem: cross ratio depends only on flags} is independent of $n_0,\dots,n_t$.
By Lemma \ref{lem: product of repsilon loxodromic} we have \begin{equation}\label{eq: in prop directions interior to limit cone approached by mu(Gamma) I}
\Vert \mu(\rho(\gamma))-A'-n_0\lambda(\rho(\gamma_0))-\dots-n_t\lambda(\rho(\gamma_{t}))\Vert<\delta
\end{equation}
\noindent for all $n_0,\dots,n_t\geq N$.

On the other hand, by Lemma \ref{lem: for prop directions interior to limit cone approached by mu(Gamma)} the set $$\{n_0\lambda(\rho(\gamma_0))+\dots+n_t\lambda(\rho(\gamma_{t}))\}_{n_j\geq 0}$$ \noindent is $\delta$-dense in $A+\calc$.
Hence, there is a compact subset $K\subset\liea$ so that $$\{n_0\lambda(\rho(\gamma_0))+\dots+n_t\lambda(\rho(\gamma_{t}))\}_{n_j\geq N}$$ \noindent is $\delta$-dense in $(A+\calc)\cap K^c$.
Now since $X_0\in\tn{int}(\calc)$, there is some positive $c$ such that $-A'+Y+cX_0\in( A+\calc)\cap K^c$.
We then find some $n_0,\dots,n_t\geq N$ so that $$\Vert  n_0\lambda(\rho(\gamma_0))+\dots+n_t\lambda(\rho(\gamma_{t})) -(-A'+Y+cX_0) \Vert<\delta.$$
\noindent By Equation (\ref{eq: in prop directions interior to limit cone approached by mu(Gamma) I}) we conclude $$\Vert \mu(\rho(\gamma))-Y-cX_0\Vert<2\delta.$$
\end{proof}

\subsection{Examples}\label{subsec: examples of anosov reps}

As already mentioned, the first examples of Anosov representations are convex co-compact representations into rank one Lie groups.
In the following we focus in some higher rank examples (other examples will be discussed in Subsection \ref{subsec: exampples of dods}).

\begin{ex}[Hitchin representations]\label{ex: hitchin}
For a split real Lie group $\g$ we denote by $\Lambda_\g:\PSL_2(\rr)\to\g$ the unique (up to conjugation) principal embedding, see Kostant \cite{Kos}.
For instance, when $\g=\PSL_d(\rr)$, then $\Lambda_\g$ is the unique irreducible representation of $\PSL_2(\rr)$.
A \textit{Hitchin representation} is a deformation of $$\rho:\pi_1(S)\to\PSL_2(\rr)\to\g,$$ \noindent where $S$ is a closed orientable surface of negative Euler characteristic, the first arrow is the holonomy representation corresponding to a hyperbolization of $S$, and the second arrow is given by composition with $\Lambda_\g$.
Labourie \cite{Lab} showed that Hitchin representations are $\Delta$-Anosov.
\end{ex}

\begin{ex}[Benoist representations]\label{ex: divisible convex}
Let $\g=\PSL(V)$ where $V$ is a real vector space as in Example \ref{ex: flags sld}.
A \textit{Benoist representation} is a faithful and discrete representation $\rho:\Gamma\to\g$ so that the image $\rho(\Gamma)$ acts properly and co-compactly on an open strictly convex subset $\mathcal{C}_\rho$ of the projective space $\mathbb{P}(V)$.
These are $\{\alpha_{1},\alpha_{d-1}\}$-Anosov and form connected components of the corresponding character variety \cite{Koszul,BenoistDivIII}.
They were extensively studied by Benoist \cite{BenoistDivI,BenoistDivII,BenoistDivIII,BenoistDivIV}.
\end{ex}



\section{Domains of discontinuity}\label{sec: dod}

We are now ready to prove our main result (Theorem \ref{thm: dod for anosov} below).
Let $\rho:\Gamma\to\g$ be a discrete representation.
An open $\Gamma$-invariant subset $\Omega\subset\x$ is a \textit{domain of discontinuity} for $\rho$ if for every compact subset $K\subset\Omega$ one has $$\#\{ \gamma\in\Gamma:\hspace{0,3cm} \rho(\gamma)\cdot K\cap K\neq\emptyset\} <\infty.$$
\noindent In this case, the action $\Gamma\curvearrowright\Omega$ is said to be \textit{properly discontinuous}.

We have the following alternative characterization.
Let $\Omega\subset\x$ be an open $\Gamma$-invariant subset.
Two points $x$ and $x'$ in $\Omega$ are said to be \textit{dynamically related} under $\rho(\Gamma)$ if there exist sequences $\{x_p\}\subset\Omega$ and $\{\gamma_p\}\subset\Gamma$ with $\gamma_p\to\infty$ such that
\[x_p \to x \quad \text{and} \quad \rho(\gamma_p)\cdot x_p\to x'.\]
\noindent In this case we say $x$ and $x'$ are dynamically related \textit{via} the sequence $\{\rho(\gamma_p)\}_{p\geq 0}$.

The following lemma is direct from definitions.

\begin{lem}\label{lem: prop discont vs dyn rels}
Let $\rho:\Gamma\to\g$ be a discrete representation with infinite image, and $\Omega\subset\x$ be an open $\Gamma$-invariant subset.
Then $\Omega$ is a domain of discontinuity for $\rho$ if and only if no two points in $\Omega$ are dynamically related under $\rho(\Gamma)$.
\end{lem}

\subsection{Statement and proof of the result}\label{subsec statement domains of discontinuity}

Let $\p_\theta$ be a self-opposite parabolic subgroup of $\g$, associated to some non-empty subset $\theta\subset\Delta$.
Further, assume that $\p_\theta\backslash\g/\h$ is finite.
We have the following key step, inspired from Kapovich-Leeb-Porti \cite[Proposition 6.2]{KLPdomains}.

\begin{lem}\label{lem: for domains of discontinuity}
Let $\rho:\Gamma\to\g$ be a $\theta$-Anosov representation with limit map $\xi_{\rho,\theta}:\bg\to\f_\theta$.
Let $\lbrace\gamma_p\rbrace_{p\geq 0}$ be a sequence in $\Gamma$ going to infinity and suppose that, as $p\to\infty$, $$U_\theta(\rho(\gamma_p))\to\xi_+ \tn{ and } S_\theta(\rho(\gamma_p))\to\xi_-$$
\noindent for some flags $\xi_{\pm}\in\xi_{\rho,\theta}(\bg)$.
Let $x$ and $x'$ be two points in $\x$ which are dynamically related via $\{\rho(\gamma_p)\}$.
Then for every $\bfp\in\p_\theta\backslash\g/\h$ satisfying $\bfp\leftrightarrow \pos(\xi_-,x)$, one has $\pos(\xi_+,x')\leq\bfp$.
\end{lem}

\begin{proof}
  Since $\bfp\leftrightarrow \pos(\xi_-,x)$ we can find two transverse flags $\xi_1$ and $\xi_2$ in $\f_\theta$ and a point $\widetilde{x}\in\x$ such that
  \[\pos(\xi_1,\widetilde{x})=\bfp\tn{ and }\pos(\xi_2,\widetilde{x})=\pos(\xi_-,x).\]
  Since the action of $\g$ on a fixed relative position is transitive, we find an element $g\in\g$ such that $g\cdot(\xi_-,x)=(\xi_2,\widetilde x)$.
  The element $\xi':=g^{-1}\cdot \xi_1$ is therefore transverse to $\xi_-$ and satisfies $\pos(\xi',x)=\bfp$.

  Take a sequence $x_p\to x$ in $\x$ such that $\rho(\gamma_p)\cdot x_p\to x'$ and write $x_p=g_p\cdot x$, for some sequence $g_p\to 1$ in $\g$.
  As $\xi'$ is transverse to $\xi_-$, so is $g_p\cdot\xi'$ for large enough $p$.
  By Remark \ref{rem: dynamics g cartan attractor and repellor} and Equation (\ref{eq def anosov}) we conclude
  \[\rho(\gamma_p) g_p\cdot\xi'\to \xi_+.\]
  Hence
  \[\pos(\xi_+,x')\leq\pos(\rho(\gamma_p) g_p\cdot\xi',\rho(\gamma_p) g_p\cdot x)=\pos(\xi',x)=\bfp.\qedhere\]
\end{proof}

Given an ideal $\bfi\subset\p_\theta\backslash\g/\h$ we define \begin{equation}\label{eq domain for ideal I}
\bfo_\rho^\bfi:=\x\setminus\displaystyle\bigcup_{z\in\bg}\lbrace x\in\x: \pos(\xi_{\rho,\theta}(z),x)\in\bfi \rbrace.
\end{equation}
\noindent Note that $\bfo_\rho^\bfi$ is $\Gamma$-invariant and open (but it could be empty).

\begin{thm}\label{thm: dod for anosov}
  Let $\theta\subset\Delta$ be non-empty and so that $\iota(\theta)=\theta$, and $\rho:\Gamma\to\g$ be a $\theta$-Anosov representation.
  Suppose that $\p_\theta\backslash\g/\h$ is finite and let $\bfi \subset \p_\theta\backslash\g/\h$. 
  Further, assume one of the following holds:
\begin{enumerate}
\item $\bfi$ is a fat ideal,
\item or $\tau(\h)=\h$ for a Cartan involution $\tau$ of $\g$, and $\bfi$ is a $w_0$-fat ideal.
\end{enumerate}
Then $\bfo_\rho^\bfi$ is a domain of discontinuity for $\rho$.
\end{thm}

\begin{proof}
  In either case, suppose by contradiction that $x$ and $x'$ are two points in $\bfo_\rho^\bfi$ which are dynamically related via some sequence $\{\rho(\gamma_p)\}$, with $\gamma_p\to\infty$.
  By Proposition \ref{prop: limit set for anosov with cartan attractors} we may assume
  \[U_\theta(\rho(\gamma_p))\to\xi_+ \tn{ and } S_\theta(\rho(\gamma_p))\to\xi_-\]
  for some flags $\xi_{\pm}\in\xi_{\rho,\theta}(\bg)$.
  We denote $\bfp_- \coloneqq \pos(\xi_-,x)$ and $\bfp_+ \coloneqq \pos(\xi_+,x')$.
  Then both $\bfp_- \not\in \bfi$ and $\bfp_+ \not\in \bfi$.
  Lemma \ref{lem: for domains of discontinuity} states that $\bfp \leftrightarrow \bfp_-$ implies $\bfp_+ \leq \bfp$ for any relative position $\bfp$.
  We proceed differently in the two cases.
\begin{enumerate}
\item 
  Since $\bfi$ is fat and $\bfp_- \not\in \bfi$, we find an element $\bfp\in\bfi$ such that $\bfp\leftrightarrow \bfp_-$. By Lemma \ref{lem: for domains of discontinuity} we have $\bfp_+ \leq\bfp$, and therefore $\bfp_+ \in \bfi$. This is the desired contradiction.
\item
We claim that $\bfp_+$ is minimal.
Indeed, let $\bfp_{\min}\leq \bfp_-$ be a minimal position.
Then $w_0 \cdot \bfp_{\mathrm{min}} \leftrightarrow \bfp_{\mathrm{min}}$ by Corollary \ref{cor: w0 action and transversely related} and $w_0 \cdot \bfp_{\mathrm{min}} \leftrightarrow \bfp_-$ by Corollary \ref{cor: increasing position keeps the relation}.
Then Lemma \ref{lem: for domains of discontinuity} implies $\bfp_+ \leq w_0\cdot\bfp_{\min}$.
Since by Proposition \ref{prop: action of w0} we know that $w_0\cdot\bfp_{\min}$ is minimal, the claim follows.
By an analogous argument with $\gamma_p$ replaced by $\gamma_p^{-1}$ we find that $\bfp_-$ is also minimal.

Knowing that $\bfp_\pm$ are minimal we can obtain a contradiction.
Indeed, applying Corollary \ref{cor: w0 action and transversely related},  Lemma \ref{lem: for domains of discontinuity} and Proposition \ref{prop: action of w0}, we conclude $\bfp_- =w_0\cdot\bfp_+$.
But this is impossible as $\bfp_\pm \not\in \bfi$ and $\bfi$ is $w_0$-fat.\qedhere
\end{enumerate}
\end{proof}

We record two consequences of Theorem \ref{thm: dod for anosov}.
For a point $x\in\x$ we let $\h^x$ be its stabilizer in $\g$ and $\mathscr{M}_\theta^x$ the union of open orbits of the action $\h^x\curvearrowright\f_\theta$.
Then $\xi \in \mathscr{M}_\theta^x$ is equivalent to $\pos(\xi,x)$ being maximal.
Proposition \ref{prop: non maximal is fat} and Theorem \ref{thm: dod for anosov} imply the following.

\begin{cor}\label{cor: if limit set contained in open orbits, then dod}
Let $\p_\theta$ be a self-opposite parabolic subgroup of $\g$ so that $\p_\theta\backslash\g/\h$ is finite and contains more than one point.
Then for every $\theta$-Anosov representation $\rho:\Gamma\to\g$ the set
\[\bfo_\rho^{\Inonmax} = \{x\in\x \colon \xi_{\rho,\theta}(\bg)\subset\mathscr{M}_\theta^x\}\]
is a domain of discontinuity for $\rho$.
\end{cor}

We also have the following, which follows from Definition \ref{dfn: fat and w0fat ideal} and Theorem \ref{thm: dod for anosov}.

\begin{cor}\label{cor: dod if unique minimal position}
  Assume that $\g$ admits a Cartan involution $\tau$ satisfying $\tau(\h) = \h$, and let $\p_\theta \subset \g$ be a self-opposite parabolic subgroup such that $\p_\theta\backslash\g/\h$ is finite.
  Let $\Imin$ be the ideal consisting of all minimal positions in $\p_\theta \backslash \g / \h$ (note that it is not necessarily a minimal $w_0$-fat ideal).
Then for every $\theta$-Anosov representation $\rho:\Gamma\to\g$ the set
\[\bfo_\rho^{\Imin} = \{x\in\x \colon \pos(\xi_{\rho,\theta}(z),x) \not\in\Imin \tn{ for all }z\in\bg\}\]
is a domain of discontinuity for $\rho$.
\end{cor}

\subsection{Examples}\label{subsec: exampples of dods}
Let us discuss some examples where the previous results apply.

\begin{ex}\label{ex: dod group manifolds}
Let $\g_0$ be a linear, connected, finite center, semisimple Lie group without compact factors and $\g:=\g_0\times\g_0$.
Let $\x=\g_0$ be the corresponding group manifold (c.f. Example \ref{ex: relative positions in group manifolds}).

Let $\theta_{\tn{L}}$ and $\theta_{\tn{R}}$ be non-empty self-opposite sets of roots in $\Delta_0$.
For $\varepsilon\in\{\tn{L},\tn{R}\}$, let $\rho_{\varepsilon}:\Gamma\to\g_0$ be $\theta_\varepsilon$-Anosov with limit map $\xi_{\rho_\varepsilon,\theta_\varepsilon}$.
Then the representation $\rho:=(\rho_{\tn{L}},\rho_{\tn{R}}):\Gamma\to\g$ is $\theta$-Anosov, where $\theta:=(\theta_{\tn{L}},\theta_{\tn{R}})$.
Furthermore, its limit map is $\xi_{\rho,\theta}=(\xi_{\rho_{\tn{L}},\theta_{\tn{L}}},\xi_{\rho_{\tn{R}},\theta_{\tn{R}}})$.
  
By Example \ref{ex: relative positions in group manifolds} and Corollary \ref{cor: dod if unique minimal position} the set
\begin{equation}\label{eq: dod for group manifold}
\lbrace g\in\g_0 \colon\quad \pos(\xi_{\rho_{\tn{L}},\theta_{\tn{L}}}(z) ,g\cdot\xi_{\rho_{\tn{R}},\theta_{\tn{R}}}(z))\tn{ is not minimal for all } z\in\bg\rbrace
\end{equation}
is a domain of discontinuity for $\rho$ in $\x=\g_0$.
In particular, when $\theta_{\tn{L}}=\theta_{\tn{R}}$ we get that $$\lbrace g\in\g_0:\hspace{0,3cm} g\cdot\xi_{\rho_{\tn{R}},\theta_{\tn{R}}}(z)\neq \xi_{\rho_{\tn{L}},\theta_{\tn{L}}}(z) \tn{ for all } z\in\bg\rbrace
$$ \noindent is a domain of discontinuity for $\rho$.
It is not hard to construct examples for which this set is non-empty.

We note that the $\rho$-action of $\Gamma$ may (or may not) be proper on $\g_0$, depending on the particular example (see Guéritaud-Guichard-Kassel-Wienhard \cite[Theorem 7.3]{GGKW}).
\end{ex}

\begin{ex}\label{ex: dod Hpq}
Let $\g=\mathsf{PSO}_0(p,q)$ and take $\x=\mathbb{H}^{p,q-1}$ as in Example \ref{ex: relative positions in hpq}.
Let $\p_1^{p,q}$ be the stabilizer in $\g$ of an isotropic line. 

An important class of $\p_1^{p,q}$-Anosov representations into $\g$, coined $\mathbb{H}^{p,q-1}$-\textit{convex co-compact}, was introduced and studied by Danciger-Gu\'eritaud-Kassel \cite{DGK1}.
We refer to \cite{DGK1} for precise definitions but let us mention here that, as shown in \cite{DGK1}, these representations always admit the following non-empty domain of discontinuity inside $\mathbb{H}^{p,q-1}$: $$\Omega:=\{x\in\mathbb{H}^{p,q-1}: x\notin \xi_\rho(z)^\perp \tn{ for every } z\in\bg\},$$ \noindent where $\xi_\rho:\bg\to\partial\mathbb{H}^{p,q-1}$ is the limit map of $\rho$.
Our construction of domains of discontinuity recovers this (c.f. Example \ref{ex: relative positions in hpq} and Corollary \ref{cor: dod if unique minimal position}).
\end{ex}

\begin{ex}\label{ex: dod quadratic forms projetive anosov}
Let $p$ and $q$ be positive integers so that $d:=p+q>2$.
Let $\x\cong \PSL_d(\rr)/\mathsf{PSO}(p,q)$ be the space of quadratic forms of signature $(p,q)$ on $\rr^d$, considered up to scaling.
Let also $\theta:=\{\alpha_1,\alpha_{d-1}\}$.
By Example \ref{ex: fat ideals quadratic forms and projective} and Corollary \ref{cor: dod if unique minimal position}, for every $\{\alpha_1,\alpha_{d-1}\}$-Anosov representation $\rho$ with limit map $(\xi_\rho^1,\xi_\rho^{d-1})$, the set 
\begin{equation}\label{eq: dod quadratic forms projetive anosov}
\{x\in\x: \xi^{d-1}_\rho(z)\neq  (\xi^{1}_\rho(z))^{\perp_x} \tn{ for all } z\in\bg\}
\end{equation}
\noindent is a domain of discontinuity for $\rho$ (here $\cdot^{\perp_x}$ denotes the orthogonal complement with respect to the form $x$).
Observe that the $\Gamma$-action on $\x$ is not proper in general (c.f. \cite[Corollary 1.9]{GGKW}).

For every $ p$ and $ q$ the domain (\ref{eq: dod quadratic forms projetive anosov}) is non-empty for a Benoist representation as in Example \ref{ex: divisible convex} (because the limit set is contained in an affine chart).
It is also non-empty for a Hitchin representation as in Example \ref{ex: hitchin}.

\end{ex}



\begin{ex}\label{ex: dod complementary subspaces projective}
Let $\x\cong\SL_d(\rr)/\mathsf{S}(\mathsf{GL}_p(\rr)\times\mathsf{GL}_q(\rr))$ with $1\leq p\leq q$ such that $d=p+q$.
We have $\mu(\h)=\liea^+$ and by Benoist-Kobayashi's Theorem \ref{thm: benoist kobayashi} no infinite discrete subgroup of $\g$ acts properly on $\x$.
This applies in particular to (images of) Anosov representations, hence \cite[Corollary 1.9]{GGKW} does not apply to this example.

Suppose that $\rho$ is $\{\alpha_1,\alpha_{d-1}\}$-Anosov with limit map $(\xi_\rho^1,\xi_\rho^{d-1})$.
Then every $w_0$-fat ideal in $\p_{\{\alpha_1, \alpha_{d-1}\}} \backslash \g / \h$ gives us a domain of discontinuity for $\rho$ in $\x$.
These ideals were discussed in Example \ref{ex: fat ideals complementary subspaces and projective}.
For instance, the ideal $\Imin$ consisting of all minimal positions corresponds to the domain of discontinuity
\[\Omega_\rho^{\Imin} = \{(U^+,U^-) \in \x \colon \xi_\rho^1(z) \pitchfork U^\pm \text{ or } \xi_\rho^{d-1}(z) \pitchfork U^\pm \text{ for all $z \in \partial\Gamma$}\},\]
where $\xi^k \pitchfork U^\pm$ means both $U^+$ and $U^-$ are transverse to $\xi^k$.

Although there is no reason for $\Omega_\rho^{\Imin}$ to be non-empty in general, it is easy to see that this is the case for Benoist representations, as well as for Hitchin representations (by a dimensional argument).

We can even find larger domains of discontinuity by using the minimal $w_0$-fat ideals from Example \ref{ex: fat ideals complementary subspaces and projective}.
That is, we have the domains $\Omega_\rho^{\bfi_{123}}$ and $\Omega_\rho^{\bfi_{234}}$ if $p > 1$, and $\Omega_\rho^{\bfi_{13}}$ and $\Omega_\rho^{\bfi_{34}}$ if $p = 1$.
All of these are supersets of $\Omega_\rho^{\Imin}$, so they are also non-empty in the case of a Benoist or a Hitchin representation $\rho$.

\end{ex}

\section{Maximality}\label{sec: maximality}

In this section we show that, under some assumptions, Theorem \ref{thm: dod for anosov} describes maximal domains of discontinuity in the case of $\Delta$-Anosov representations (Theorem \ref{thm: maximality with w0-fat} below).

\subsection{Sufficient condition for dynamical relations}

The following proposition could be thought of as a strengthening of Benoist-Kobayashi's Theorem \ref{thm: benoist kobayashi}, in the case of symmetric spaces and $\Delta$-Anosov representations acting on it.
Instead of ensuring just the existence of dynamical relations in $\x$, it provides a sufficient condition for the existence of a dynamical relation between two given points in $\x$.

\begin{prop}\label{prop suf condition for dynamical relation symmetric case}
Suppose that $\h$ is symmetric, $\liea$ is $\sigma$-invariant, and $\liea_\h:=\liea\cap\lieh$ is a Cartan subspace of $\h$.
Assume that $\liea_\h$ contains a regular element, and pick the Weyl chamber $\liea^+$ in such a way that $\tn{int}(\liea^+)\cap\liea_\h$ is non-empty.
Assume also $\m\subset\h$.
Let $\rho:\Gamma\to\g$ be a Zariski dense $\Delta$-Anosov representation so that $\mu(\h)\cap\tn{int}(\cone)\neq\emptyset$.
Pick points $\xi_-,\xi_+\in\xi_\rho(\bg)$ and $x,x'\in\x$ so that the relative positions $\pos(\xi_-,x)$ and $\pos(\xi_+,x')$ in $\bor\backslash\g/\h$ are minimal and satisfy $$\pos(\xi_-,x)\leftrightarrow\pos(\xi_+,x').$$
\noindent Then $x$ is dynamically related to $x'$ under a sequence in $\rho(\Gamma)$.
\end{prop}

\begin{proof}
We first show that our assumptions allows us to represent the relevant relative positions conveniently (Equation (\ref{eq: in prop suff cond dynamical relations}) below). 
Namely, by Matsuki's Theorem \ref{thm: matsuki minimal} and Corollary \ref{cor: related minimal positions} there is some $w\in\w^\sigma$ so that $$\pos(\xi_-,x)=\bor w_0w\h \tn{ and } \pos(\xi_+,x')=\bor w\h.$$
\noindent We may write $\xi_+=k_0\bor$ and $\xi_-=l_0^{-1}\bor^-$ for some $k_0,l_o\in\ko$.
We then find elements $b'\in\bor$ and $b^-\in\bor^-$ such that $$k_0^{-1}\cdot x' = b'w\cdot o \tn{ and }l_0\cdot x=b^-w\cdot o.$$
\noindent Now decompose $b$ and $b' $ as $b'=n'a'm'$ and $b^-=n^-am$ for some $a,a'\in\exp(\liea)$, $m,m'\in\m$, $n'\in\n$ and $n^-\in\n^-$. 
As $w$ normalizes $\m$ and $\m\subset\h$ we have \begin{equation}\label{eq: in prop suff cond dynamical relations}
k_0^{-1}\cdot x' = n'a'w\cdot o \tn{ and } l_0\cdot x=n^-aw\cdot o.
\end{equation}

We now apply Proposition \ref{prop: directions interior to limit cone approached by mu(Gamma)} and \cite[Lemma 3.1.5]{SteThesis} to show that we may find a sequence in $\Gamma$ going to infinity that allows us to ``go from $n^-a$ to $n'a'$" in the equation above.
More precisely, by Proposition \ref{prop: directions interior to limit cone approached by mu(Gamma)} there are sequences $\gamma_p\to\infty$ in $\Gamma$ and $A_p\in\mu(\h)$ such that $U(\rho(\gamma_p))\to\xi_+$, $S(\rho(\gamma_p))\to\xi_-$, and
$$d(\mu(\rho(\gamma_p)), -\log(a)+\log(a')+A_p)\to 0.$$
\noindent Let $h_p:=\exp(A_p)$, which belongs to $\h$ thanks to Corollary \ref{cor cartan compatible with H}.
We have \begin{equation}\label{eq: in prop suff cond dynamical relations I}\exp(\mu(\rho(\gamma_p))a=a_pa'h_p,
\end{equation}
\noindent for some sequence $\{a_p\}\subset\exp(\liea)$ such that $a_p\to 1$.

On the other hand, consider a Cartan decomposition $\rho(\gamma_p)=k_p\exp(\mu(\rho(\gamma_p)))l_p$ of $\rho(\gamma_p)$.
Up to taking a subsequence if necessary, we may assume $k_p\to k$ and $l_p\to l$ for some $k,l\in\ko$.
Note $$k_0\bor=\xi_+=k\bor \tn{ and } l_0^{-1}\bor^-=\xi_-=l^{-1}\bor^-.$$
\noindent There exist then $m_0,\widetilde{m}_0\in\m$ so that $km_0=k_0 \tn{ and } \widetilde{m}_0^{-1}l=l_0.$
Thus $$k^{-1}\cdot x'=(m_0k_0^{-1})\cdot x'=m_0n'a'w\cdot o.$$
\noindent Since $\m$ normalizes $\n$ we have $k^{-1}\cdot x'=\widetilde{n}'a'm_0w\cdot o$ for some $\widetilde{n}'\in\n$. Similarly, $l\cdot x=\widetilde{n}^-a\widetilde{m}_0w\cdot o$ for some $\widetilde{n}^-\in\n^-$.
Proceeding as in (\ref{eq: in prop suff cond dynamical relations}) to eliminate $m_0$ and $\widetilde{m}_0$ we conclude \begin{equation}\label{eq: in prop suff cond dynamical relations II}
k^{-1}\cdot x'=\widetilde{n}'a'w\cdot o \tn{ and } l\cdot x=\widetilde{n}^-aw\cdot o.
\end{equation}

Now since $\rho$ is $\Delta$-Anosov, we have $\alpha(\mu(\rho(\gamma_p)))\to\infty$ for all $\alpha\in\Delta$.
By \cite[Lemma 3.1.5]{SteThesis}, we may take a sequence $g_p\to \widetilde{n}^-$ so that $$\exp(\mu(\gamma_p))g_p\exp(-\mu(\gamma_p))\to \widetilde{n}'.$$
\noindent Then $x_p:=g_paw\cdot o\to \widetilde{n}^-aw\cdot o=l\cdot  x$ and $$\exp(\mu(\rho(\gamma_p)))\cdot x_p=(\exp(\mu(\gamma_p)) g_p \exp(-\mu(\gamma_p)))\exp(\mu(\gamma_p))aw\cdot o.$$
\noindent By Equation (\ref{eq: in prop suff cond dynamical relations I}) we have $$\exp(\mu(\rho(\gamma_p)))\cdot x_p=(\exp(\mu(\gamma_p)) g_p \exp(-\mu(\gamma_p))) a_pa'h_p   w\cdot o.$$
\noindent Since $w\in\w^\sigma$ and $h_p\in\exp(\liea_\h)$ we have $h_pw\cdot o=w\cdot o$.
Hence by Equation (\ref{eq: in prop suff cond dynamical relations II}) $$\exp(\mu(\rho(\gamma_p)))\cdot x_p\to \widetilde{n}'a'w\cdot o=k^{-1}\cdot x'.$$
\noindent We conclude that $l\cdot x$ is dynamically related to $k^{-1}\cdot x'$ through the sequence $\{\exp(\mu(\rho(\gamma_p)))\}$.
Hence $x$ is dynamically related to $x'$ through the sequence $\{\rho(\gamma_p)\}$.
\end{proof}

\begin{rem}\label{rem: sufficient condition dynrel flags}
Proposition \ref{prop suf condition for dynamical relation symmetric case} takes inspiration from \cite[Lemma 3.8]{SteIDEALS}, which proves a similar statement in the case $\h=\p_{\theta'}$.
That result guarantees a dynamical relation between $x$ and $x'$ in $\f_{\theta'}$ provided that $$\pos(\xi_-,x)=w_0\cdot\pos(\xi_+,x'),$$
\noindent but with no minimality assumption on these relative positions.
Note that when $\h=\p_{\theta'}$ one has $\mu(\h)=\liea^+$, therefore the condition $\mu(\h)\cap \tn{int}(\cone)\neq\emptyset$ is always satisfied.

If $\h$ is symmetric a further assumption on the relative positions $\pos(\xi_-,x)$ and $\pos(\xi_+,x')$ is needed.
For instance, consider $\x\cong\PSL_3(\rr)/\mathsf{PSO}(2,1)$.
There are two relative positions $\mathbf{p} \neq \mathbf{p}'$ which are neither minimal nor maximal, fixed by $w_0$ and transversely related (c.f. Figure \ref{fig: intro}).
We might then have $$\pos(\xi_-,x)=\bfp \leftrightarrow \bfp' = \pos(\xi_+,x').$$
\noindent But $\bfp$ is transversely related to the minimal position, so Lemma \ref{lem: for domains of discontinuity} forbids the possibility of having a dynamical relation between $x$ and $x'$.
\end{rem}

\begin{rem}\label{rem: MsubsetH and oriented}
The assumption $\m\subset\h$ in Proposition \ref{prop suf condition for dynamical relation symmetric case} might look unnatural.
However, Example \ref{ex: MsubsetH is needed for dynamical relation} below shows that it is necessary, as it showcases an example of a $\zz$-action on the group manifold $\SL_2(\rr)$ for which our domains of discontinuity are not maximal.
We mention that the same idea can be used for other word hyperbolic groups and group manifolds (e.g. $\SL_3(\rr)$) to construct domains of discontinuity which are larger than the one given by Theorem \ref{thm: dod for anosov}, as long as there is a consistent way to remove  ``refined bad" positions.
This can be done if the limit curve lifts to a finite cover $\widetilde\f$ of the flag manifold, by considering relative positions in $\g\backslash(\widetilde\f \times \x)$ instead of $\g\backslash(\f \times \x)$, similarly to the situation in \cite{SteTreORIENTED}.
\end{rem}

\begin{ex}\label{ex: MsubsetH is needed for dynamical relation}

Consider the group manifold $\x:=\SL_2(\rr)$ and a sequence $$\gamma_p:=\left(  \left(  \begin{matrix}
\mu^{\tn{L}}_p & \\
 & (\mu_p^{\tn{L}})^{-1}
\end{matrix}   \right),\left(  \begin{matrix}
\mu^{\tn{R}}_p & \\
 & (\mu_p^{\tn{R}})^{-1}
\end{matrix}   \right) \right) ,$$ \noindent where $\mu_p^{\tn{L}},\mu_p^{\tn{R}}\to\infty$.
If $e_1, e_2$ denote the standard basis vectors in $\rr^2$, then for all $p$, $$U(\gamma_p)=(\rr e_1,\rr e_1)=:\xi_+ \tn{ and } S(\gamma_p)=(\rr e_2,\rr e_2)=:\xi_-.$$

On the other hand, take two points $\left(  \begin{smallmatrix}
a & b\\
c & d
\end{smallmatrix}   \right)$ and $\left(  \begin{smallmatrix}
a' & b'\\
c' & d'
\end{smallmatrix}   \right)$ in $\x$, and suppose that they are dynamically related through $\{\gamma_p\}$.
By Lemma \ref{lem: for domains of discontinuity} we have $b=c'=0$.
Furthermore, one may directly check that the dynamical relation implies $ \mu_p^{\tn{L}}/\mu_p^{\tn{R}}\to a'/a$ as $p\to\infty$.
In particular, $a$ and $a'$ must have the same sign.
In particular, the points $x:=\left(  \begin{smallmatrix}
1 & 0\\
0 & 1
\end{smallmatrix}   \right)$ and $x':=\left(  \begin{smallmatrix}
-1 & 0\\
0 & -1
\end{smallmatrix}   \right)$ cannot be dynamically related through the sequence $\{\gamma_p\}$.
However, $\pos(\xi_-,x)\leftrightarrow\pos(\xi_+,x')$ and these two positions are minimal (and coincide).
\end{ex}

\subsection{Maximal domains of discontinuity}\label{subsec: maximality of dods}

We are now ready to prove our maximality result.

\begin{thm}\label{thm: maximality with w0-fat}
Suppose that $\h$ is symmetric, $\liea$ is $\sigma$-invariant, and $\liea_\h:=\liea\cap\lieh$ is a Cartan subspace of $\h$.
Assume that $\liea_\h$ contains a regular element, and pick the Weyl chamber $\liea^+$ in such a way that $\tn{int}(\liea^+)\cap\liea_\h$ is non-empty.
Assume also $\m\subset\h$.
Let $\rho:\Gamma\to\g$ be a Zariski dense $\Delta$-Anosov representation so that $\mu(\h)\cap\tn{int}(\cone)\neq\emptyset$.
Then if $\bfo\subset \x$ is a maximal domain of discontinuity for $\rho$, there exists a $w_0$-fat ideal $\bfi\subset\bor\backslash\g/\h$ so that $\bfo=\bfo_\rho^\bfi$.
\end{thm}

\begin{proof}
The proof follows the approach by \cite{SteIDEALS}.
Indeed, let $$\bfi:=(\bor\backslash\g/\h)\setminus\{\pos(\xi,x): \xi\in\xi_\rho(\bg) \tn{ and } x\in\bfo\},$$
\noindent and observe that $\bfo\subset\bfo_\rho^\bfi$.
Hence by the maximality assumption and Theorem \ref{thm: dod for anosov}, it suffices to show that $\bfi$ is a $w_0$-fat ideal.

We first show that $\bfi$ is an ideal.
To do so, we suppose by contradiction that there is some $\bfp\in\bfi$ and $\bfp'\leq \bfp$ so that $\bfp'\notin\bfi$.
We may then represent $\bfp'=\pos(\xi_0,x_0)$ for some $\xi_0\in\xi_\rho(\bg)$ and $x_0\in\bfo$.
Now since $\bfp\in\bfi$, the set $$\{x\in\x: \pos(\xi_0,x)=\bfp\}$$
\noindent is contained in $\x\setminus\bfo$, which is closed.
Hence $$\overline{\{x\in\x: \pos(\xi_0,x)=\bfp\}}\subset \x\setminus\bfo.$$
\noindent But since $\bfp'\leq\bfp$, we have $$\{x\in\x: \pos(\xi_0,x)=\bfp'\}\subset\overline{\{x\in\x: \pos(\xi_0,x)=\bfp\}}.$$
\noindent This implies $x_0\notin\bfo$, a contradiction.

To show that $\bfi$ is $w_0$-fat we proceed again by contradiction.
Suppose then that there is some $\bfp_{\min}\notin\bfi$ such that $w_0\cdot\bfp_{\min}\notin\bfi$.
We may then write $$\bfp_{\min}=\pos(\xi_-,x) \tn{ and } w_0\cdot\bfp_{\min}=\pos(\xi_+,x')$$
\noindent for some $x,x'\in\bfo$ and $\xi_{\pm}\in\xi_\rho(\bg)$.
Proposition \ref{prop suf condition for dynamical relation symmetric case} guarantees the existence of a dynamical relation between $x$ and $x'$ under a sequence in $\rho(\Gamma)$, finding the desired contradiction.
\end{proof}

For the space of complementary subspaces on a real vector space we can fully apply Theorem \ref{thm: maximality with w0-fat}.

\begin{ex}\label{ex: maximal complementary subspaces}
  Let $\x \cong \g / \h = \SL_d(\rr)/\mathsf{S}(\mathsf{GL}_p(\rr)\times\mathsf{GL}_q(\rr))$ for $p\geq 1$ and $q\geq 1$ such that $d=p+q$.
  Note that $\mu(\h)=\liea^+$ and $\m\subset\h$.
For convenience, we will write $\p_1$ for the parabolic subgroup $\p_{\{\alpha_1, \alpha_{d-1}\}}$, while $\bor=\p_\Delta$ is the Borel subgroup.

Recall from Example \ref{ex: fat ideals complementary subspaces and projective} and Example \ref{ex: dod complementary subspaces projective} that there are two minimal $w_0$-fat ideals in $\p_1 \backslash \g / \h$, corresponding to two domains of discontinuity in $\x$ for every $\{\alpha_1, \alpha_{d-1}\}$-Anosov representation $\rho \colon \Gamma \to \SL_d(\rr)$.

Now we want to consider the action of a $\Delta$-Anosov representation $\rho \colon \Gamma \to \SL_d(\rr)$ instead.
Note that $\rho$ is in particular $\{\alpha_1,\alpha_{d-1}\}$-Anosov, so the domains from Example \ref{ex: dod complementary subspaces projective} are still domains of discontinuity.
However, we can find larger ones.

To do this, consider the natural surjection $q \colon \bor \backslash \g / \h \to \p_1 \backslash \g / \h$.
If we represent elements of $\bor \backslash \g / \h$ as subsets $A \subset \{1, \dots, d\}$ with $\#A = p$ as in Example \ref{ex: matsuki and w0 action complementary subspaces}, and let the minimal positions in $\p_1 \backslash \g / \h$ be $\{\bfp_1, \bfp_2, \bfp_3, \bfp_4\}$ as in Example \ref{ex: fat ideals complementary subspaces and projective}, then
\[q^{-1}(\bfp_1) = \{A \mid 1 \in A \land d \not\in A\}, \quad q^{-1}(\bfp_2) = \{A \mid 1 \in A \land d \in A\},\]
\[q^{-1}(\bfp_3) = \{A \mid 1 \not\in A \land d \not\in A\}, \quad q^{-1}(\bfp_4) = \{A \mid 1 \not\in A \land d \in A\}.\]
Recall that the $w_0$-action fixes $\bfp_2$ and $\bfp_3$ while it interchanges $\bfp_1$ and $\bfp_4$.
Hence we had to include $\bfp_2$ and $\bfp_3$ in every $w_0$-fat ideal of $\p_1 \backslash \h / \g$.
The $w_0$-action on $\bor \backslash \g / \h$ accordingly maps $q^{-1}(\bfp_1)$ to $q^{-1}(\bfp_4)$ and maps the sets $q^{-1}(\bfp_2)$ and $q^{-1}(\bfp_3)$ to themselves, but does not fix every element of them (we saw in Example \ref{ex: matsuki and w0 action complementary subspaces} that it has no fixed points if $p$ and $q$ are both odd, and $\binom{\lfloor d/2 \rfloor}{\lfloor p/2 \rfloor}$ fixed points otherwise).

This gives us some choices to pick a minimal $w_0$-fat ideal $\bfi \subset \bor \backslash \g / \h$.
For instance, we could include all of $q^{-1}(\bfp_1)$ and none of $q^{-1}(\bfp_4)$, as well as from $q^{-1}(\bfp_2)$ and $q^{-1}(\bfp_3)$ exactly one of each pair of positions identified by $w_0$ (as well as all fixed points of $w_0$).
For any choice like this, $\Omega_\rho^\bfi \subset \x$ will be a domain of discontinuity for $\rho$ which contains the domain $\Omega_\rho^{\bfi_{123}}$ from Example \ref{ex: dod complementary subspaces projective}.
Theorem \ref{thm: maximality with w0-fat} applies in this situation and tells us that $\Omega_\rho^\bfi$ is in fact maximal, i.e. there is no strictly larger open subset of $\x$ on which $\rho$ can act properly discontinuously.

Similarly, we can construct maximal domains of discontinuity containing $\Omega_\rho^{\bfi_{234}}$ from Example \ref{ex: dod complementary subspaces projective}, as well as ones which are independent of $\Omega_\rho^{\bfi_{123}}$ and $\Omega_\rho^{\bfi_{234}}$.

This discussion give us non-empty maximal open domains of discontinuity in $\x$ e.g. for every Zariski dense Hitchin representation.

\end{ex}

\bibliographystyle{plain}
\bibliography{dynamical_relations}
\end{document}